\documentclass[reqno]{amsart}

\usepackage{mathrsfs}

\usepackage{amssymb}
\usepackage{enumitem}
\usepackage{thm-restate}
\usepackage{tikz}
\usepackage{undertilde}
\usepackage[backref, hypertexnames=false]{hyperref}
\usepackage{stmaryrd}
\usepackage{mdframed}
\usepackage{nicefrac}

\usepackage{todonotes}
\usepackage{enumitem}
\usepackage{orcidlink}

\theoremstyle{definition}
\newtheorem{thm}{Theorem}[section]
\newtheorem{cor}[thm]{Corollary}
\newtheorem{prop}[thm]{Proposition}

\newtheorem{case}[thm]{Case}
\newtheorem{defn}[thm]{Definition}
\newtheorem{ex}[thm]{Example}
\newtheorem{rmk}[thm]{Remark}

\newtheorem*{claim*}{Claim}

\numberwithin{equation}{section}

\newcommand{\aca}{\mathsf{ACA_0^{-}}}

\newcommand{\schoice}{\mathsf{\Sigma^1_1\mbox{-}AC_0^{-}}}
\newcommand{\schoicewo}{\mathsf{\Sigma^1_1\mbox{-}AC_0^{wo}}}
\newcommand{\dcomp}{\mathsf{\Delta^1_1\mbox{-}CA_0^{-}}}

\newcommand{\scomp}{\mathsf{\Sigma^1_1\mbox{-}CA_0^{-}}}

\newcommand{\scompfin}{\mathsf{\Sigma^1_1\mbox{-}CA_0^{fin}}}
\newcommand{\sol}{\mathsf{SOL}}

\usepackage{comment}
\specialcomment{proofdetail}{\color{blue}}{\color{black}}

\excludecomment{proofdetail}

\newcommand{\nocontentsline}[3]{}
\newcommand\stoptoc{%
  \let\origcontentsline\addcontentsline
  \let\addcontentsline\nocontentsline
}
\newcommand\resumetoc{%
  \let\addcontentsline\origcontentsline
}

\begin{document}

\newpage 

\title[Natural language quantifiers are predicatively definable]{The generalized quantifiers of natural language are predicatively definable}

\author{Sean Walsh\,\orcidlink{0009-0007-5460-778X}}

\address{Department of Philosophy \\ University of California, Los Angeles \\ 390 Portola Plaza, Dodd Hall 321 \\ Los Angeles, CA 90095-1451}

\email{walsh@ucla.edu}

\urladdr{http://philosophy.ucla.edu/person/sean-walsh/}


\date{}

\dedicatory{}

\begin{abstract}
\makeatletter\phantomsection\def\@currentlabel{(abstract)}\makeatother
This paper studies the definability of natural language generalized quantifiers. The semantics of generalized quantifiers are provided by a collection of subsets of the underlying domain. However, the generalized quantifiers appearing in natural language are definable either by first-order quantification or by cardinality notions. This paper provides an explanation for this observed phenomenon. The explanation is that the famous constraints of domain independence and conservativity, when extended to Henkin models, suffice to ensure low-level definability, namely $\Delta^1_1$-definability or at least $\Sigma^1_1$-definability; and in most cases this definability can be made to be bounded. This is basically a consequence of Feferman's Preservation Theorem~\cite{Feferman1968-fi}, which Marker \cite{Marker1984aa} has provided a short model-theoretic proof of. Further, we verify that the paradigmatic cardinality quantifiers are indeed $\Delta^1_1$-definable for a reasonable choice of background theory. Finally, in many other cases, we show that this definability can be lowered to first-order definability.
\end{abstract}

\maketitle

\tableofcontents

\section{Introduction}\label{sec:intro}

Generalized quantifiers are a major topic of study at the intersection of semantics and logic.\footnote{The standard references are Keenan-Westerst{\aa}hl \cite{Keenan1997-vc}, Westerst{\aa}hl \cite{Westerstahl2007-fb}, Peters-Westerst{\aa}hl \cite{Peters2008aa}, and Szabolcsi \cite{Szabolcsi2010-vi}. See these references for more of the history, but much of the work originated in Barwise-Cooper \cite{Barwise1981-ur} and Keenan-Stavi \cite{Keenan1986-pq}. Generalized quantifiers were studied earlier in mathematical logic by Mostowski \cite{Mostowski1957-yk}; see Feferman-Barwise \cite{Barwise1985-ys} for more on this tradition. \label{fn:one}} Linguists have isolated several semantic constraints, like domain independence and conservativity, which appear to hold of all generalized quantifiers which appear in natural language. This paper shows that these constraints, when extended to Henkin models, have deep consequences for how difficult it is to define these quantifiers: namely, these constraints require that the quantifiers are $\Delta^1_1$-definable, or in some limiting cases $\Sigma^1_1$-definable; and further in many cases this definability can be made to be bounded (see Theorems~\ref{thm:fefermanesq}, \ref{thm:fefermanesqfour}, \ref{thm:fodefinable}, \ref{thm:fodefinabletwo} below). As we will see, this is basically a consequence of Feferman's Preservation Theorem,\footnote{\cite{Feferman1968-fi}. As Feferman discusses \cite[32]{Feferman1968-fi}, his results generalized the earlier-announced results \cite{Feferman1966-tz} which were joint with Kreisel.} which Marker  has provided a short model-theoretic proof of.\footnote{\cite{Marker1984aa}.}

This is important because if one accepts these semantic constraints of domain independence and conservativity and their extensions to Henkin models, then these Theorems entail that natural language quantifiers are low-level definable, and this explains why so few of the theoretically possible options for generalized quantifiers are realized in natural language, and why they all seem to expressible with simple combinations of  first-order existential quantifiers, first-order universal quantifiers, and cardinality notions. Or, if one takes the task to be to explain both (i)~the semantic constraints of domain independence and conservativity, and (ii)~why so few of the theoretically possible options for generalized quantifiers are realized in natural language, then this paper is important because it shows that a solution to~(i) provides a solution to~(ii), once one extends to Henkin models. 

Finally, the results of this paper are important because on finite models, low-level second-order definability is highly correlated with grades of infeasibility by Fagin's Theorem; there is a similar lesson for Borel sets from Souslin's Theorem, the setting of much ordinary infinitary mathematics (cf. Corollary~\ref{cor:logicapps}). Hence the mathematical theorems in this paper isolate upper bounds on the complexity of natural language quantifiers under two natural measures of complexity. This forms a natural complement to empirical work which tries to assay more precisely the extent of the feasibility and learnability of natural language quantifiers.\footnote{\cite{Hunter2013-qv}, \cite{Szymanik2016-ue}, \cite{Steinert-Threlkeld2020-mm}.}

There are two more minor reasons why these results are important. First, they suggests a certain moderation is in order when discussing the significance of non-first-definability of some generalized quantifiers. To be sure, this is important, and no doubt was more important in earlier periods when first-order logic occupied a more dominant place in philosophical and logical theorizing. But the theorems tells us that the definability is to be had at the immediate next stage beyond first-order logic. Second, the theorems lends plausibility to the idea that approaches to formal semantics are possible even in the absence of full comprehension.\footnote{See (\ref{eqn:compschema}) for a formal statement of the comprehension schema. Intuitively it says that formulas with free first-order variables determine second-order entities.}$^{\;,}$\footnote{The paper \cite{Walsh2016-uo} sets out a predicative intensional theory of types on which one could build such an approach.} For, at least in introductory treatments of formal semantics, the only items in the lexicon whose logical forms have any logical complexity are the generalized quantifiers.\footnote{This is so at least for the lexicon used in \cite[Chapter 6]{Gamut1991-gc}. It is not presently clear to me whether this holds more generally. This paper grew out of an attempt to understand whether the lexicon in \cite[Chapter 6]{Gamut1991-gc} had denotations in the models studied in \cite{Walsh2016-uo}. The first question in \S\ref{sec:further} records that, in this, I have at best met with partial success.} The main Theorems~\ref{thm:fefermanesq}, \ref{thm:fefermanesqfour}, \ref{thm:fodefinable}, \ref{thm:fodefinabletwo} tell us that the amount of comprehension we need for the generalized quantifiers is limited.

This paper is organized as follows. In \S\ref{sec:henkin} we introduce Henkin models and notation for them. In \S\ref{sec:genq} we define generalized quantifiers in a way that applies to Henkin models and recall some paradigmatic examples of generalized quantifiers. In \S\S\ref{sec:henkinconstraints:1}-\ref{sec:henkinconstraints:2} we describe how to define domain independence and conservativity so that they apply to Henkin models. In \S\ref{sec:disjunctiveexamples} we describe some disjunctive examples which are ruled out by domain independence. In \S\ref{sec:defnandboundeddefn}, we formalize our predicative definability and bounded definability notions and state some basic results on absoluteness of these definitions (Theorem~\ref{thm:stable}, Theorem~\ref{thm:internal}). In \S\ref{sec:statement} we state our two main Theorems (Theorems~\ref{thm:fefermanesq}, \ref{thm:fefermanesqfour}). In \S\ref{sec:discuss} we discuss the significance of our results for explanatory projects in formal semantics. In \S\ref{sec:discuss:2} we argue that the motivations for domain independence and conservativity extend from standard models to Henkin models.  In \S\ref{sec:lower} we describe how sometimes first-order definability rather than $\Delta^1_1$-definability can be achieved, and we state Theorems~\ref{thm:fodefinable}, \ref{thm:fodefinabletwo}. In \S\ref{sec:cardinality} we identify a background theory which is apt to handle quantifiers whose truth-conditions are given by recourse to cardinality notions (cf. Theorem~\ref{thmmostdefinabel}). In \S\ref{sec:applications} we discuss the aforementioned applications to computational complexity and descriptive set theory and note how Corollary~\ref{cor:logicapps} follows immediately from Theorem~\ref{thm:fefermanesq}.

The remainder of the paper is more technical and in \S\S\ref{sec:stable}-\ref{sec:cardinalityqs} we turn to the detailed work of proving the theorems. Namely, in \S\ref{sec:stable} we prove Theorem~\ref{thm:stable} and Theorem~\ref{thm:internal}. In \S\ref{sec:henkinish} we introduce Marker's notion of faithful embeddings and relate this to Henkin models as well as to the models we get out of the Compactness Theorem. In \S\ref{sec:proofmain} we prove Theorems~\ref{thm:fefermanesq}, \ref{thm:fefermanesqfour}, in \S\ref{sec:fodefinable} we prove Theorems~\ref{thm:fodefinable}-\ref{thm:fodefinabletwo}, and in \S\ref{sec:cardinalityqs} we prove Theorem~\ref{thmmostdefinabel}. Finally, in \S\ref{sec:further} we remark on a few directions for further research.

\section{Henkin models of second-order logic and its subsystems}\label{sec:henkin}

We work with Henkin models of second-order logic and its subsystems. These are an important class of models since the Completeness Theorem holds for them and their validity relation is recursively enumerable.\footnote{See any treatment of second-order logic, such as \cite{Manzano1996-ff}, \cite{Enderton2001-bm}.} In this section, we introduce notation for Henkin models.\footnote{While the main results of this paper are all explicable using Henkin models, the proofs of these results require fluidity in moving back and forth between Henkin models and a more general class of models, which we review in \S\ref{sec:henkinish}.} Readers familiar with second-order logic may safely skip this section and come back to it as needed.

Henkin models have the form $\mathcal{A}=(A, S_1[A], S_2[A], \ldots)$ where $A$ is a non-empty set and $S_n[A]$ is a non-empty collection of subsets of the $n$-th Cartesian product $A^n$ of $A$. The idea is that $A$ is the first-order domain and that $S_n[A]$ is the domain reserved for the $n$-ary second-order entities. The membership relation is stipulated to be interpreted absolutely, i.e. $\mathcal{A}\models R\overline{a}$ iff $\overline{a}$ is in $R$, for an element $R$ of $S_n[A]$ and an $n$-tuple $\overline{a}=(a_1, \ldots, a_n)$ from $A^n$. A Henkin model $\mathcal{A}=(A, S_1[A], S_2[A], \ldots)$ is \emph{standard} if for all $n\geq 1$ one has $S_n[A]=P(A^n)$, i.e. the interpretation $S_n[A]$ of the range of quantifiers over $n$-ary relation is the full powerset of the $n$-th Cartesian product $A^n$ of the underlying first-order domain~$A$, as judged by the ambient metatheory. The standard models satisfy all instances of the following
\emph{comprehension schema}:
\begin{equation}\label{eqn:compschema}
\exists \; R \; \forall \; \overline{x} \; (R\overline{x} \leftrightarrow \varphi(\overline{x}))
\end{equation}
wherein $\varphi(x_1, \ldots, x_n)$ has free first-order variables $x_1, \ldots, x_n$, along with perhaps other free variables (called \emph{parameters}), and wherein the $n$-ary second-order variable $R$ does not appear free in $\varphi(x_1, \ldots, x_n)$.

Some minimal set-theoretic apparatus is available in any Henkin model $\mathcal{A}$ which satisfies the comprehension schema (\ref{eqn:compschema}) for formulas $\varphi$ which do not contain second-order quantifiers. First, $S_n[A]$ is a Boolean algebra using the obvious formulas: e.g. treating $X,Y$ in $S_1[A]$ as parameters and using the formula $Xx\wedge Yx$, we have a $Z$ such that: for all $x$ in $A$, one has $Zx$ iff $Xx\wedge Yx$; and we write this as $Z=X\cap Y$. We likewise freely use the operations $X\cup Y, X\setminus Y, X\times Y, X^n$ with their usual meaning in the object-language, and for each $n\geq 1$ we use $\emptyset$ for the empty $n$-ary relation (and let context determine which $n\geq 1$ is intended). To sometimes avoid having to display arities, if $R$ is $n$-ary and $X$ is unary, we stipulate 
\begin{equation}\label{eqn:conventionpowers}
R\cap X := R\cap X^n
\end{equation}
To this we add the projection functions: for $n$-ary $R$ and $1\leq i\leq n$ we have:
\begin{equation}
\pi_i R =\{ x_i : \exists \; x_1, \ldots, x_{i-1}, x_{i+1}, \ldots, x_n \; R(x_1, \ldots, x_{i-1}, x_i, x_{i+1}, \ldots, x_n) \}
\end{equation}
as well as the field $\mathscr{U} R = \bigcup_{i=1}^n \pi_i R$ of the $n$-ary relation $R$, 
where the scripted $\mathscr{U}$ is a mnemonic for `union.' We similarly define, for a tuple $\overline{R}=(R_1, \ldots, R_m)$, the union of its fields $\mathscr{U} \overline{R}=\bigcup_{j=1}^m \mathscr{U} R_j$. Further, for each $(n+m)$-ary $R$ and tuple $\overline{x}$ of length $n$, we have the function  $(R,\overline{x})\mapsto R[\overline{x}]$ given by $R[\overline{x}]:=\{\overline{y}: R\overline{x}\,\overline{y}\}$.\footnote{This is used in the definition of the theory $\schoice$. See (\ref{eqn:choice}) below.}  In addition to these operations, we use the subset relation $R\subseteq R^{\prime}$ with its usual meaning.

A further aspect of this minimal set theoretic apparatus is that it is absolute between Henkin models and the metatheory and hence also between two Henkin models, a theme to which we will return to in~\S\ref{sec:defnandboundeddefn}.

\begin{prop}\label{prop:absolutenboolean} (Absoluteness of minimal set theoretic apparatus).

Suppose that $\mathcal{A}$ is a Henkin model which satisfies the Comprehension Schema~(\ref{eqn:compschema}) for all formulas with no second-order quantifiers. 

Suppose that $X,Y,Z$ are in $S_1[A]$. Then one has $X\subseteq Y$ iff $\mathcal{A}\models X\subseteq Y$, and one has $X\cap Y=Z$ iff $\mathcal{A}\models X\cap Y=Z$.

The same result holds for $S_n[A]$ for $n>1$ and for the operations $X\cup Y$, $X\setminus Y$, $X\times Y$, $X^n$, $\pi_i R$, $\mathscr{U}\overline{R}$, and $R[\overline{x}]$.
\end{prop}

\begin{proof}
For subset, we argue as follows:
\begin{itemize}[leftmargin=*]
    \item First suppose $X\subseteq Y$. Then for all $x$, if $x$ in $X$, then $x$ in $Y$. Then for all $x$ in $A$, if $x$ in $X$, then $x$ in $Y$. Then $\mathcal{A}\models X\subseteq Y$.
    \item Second suppose $\mathcal{A}\models X\subseteq Y$. Then for all $x$ in $A$, if $x$ in $X$, then $x$ in $Y$. Since $X$ is a subset of $A$, this implies $X\subseteq Y$.    
\end{itemize}
For intersection, we argue as follows:
\begin{itemize}[leftmargin=*]
    \item First suppose that $X\cap Y=Z$. Then for all $x$: $x$ is in both $X,Y$ iff $x$ in $Z$. Then for all $x$ in $A$: $x$ is in both $X,Y$ iff $x$ in $Z$. Then $\mathcal{A}\models X\cap Y=Z$.
    \item Second suppose $\mathcal{A}\models X\cap Y=Z$. Then for all $x$ in $A$: $x$ is in both $X,Y$ iff $x$ in $Z$. Since $X,Y,Z$ are subsets of $A$, we have that for all $x$: $x$ is in both $X,Y$ iff $x$ in $Z$. Then $X\cap Y=Z$.    
\end{itemize}
The proofs for $n>1$ and the other operations listed are similar.
\end{proof}

Due to their absoluteness and their definability in systems which have comprehension for formulas without second-order quantifiers, we may regard ourselves as working in a definitional expansion of second-order logic which has constant, relation, and function symbols for all this minimal set-theoretic apparatus, namely: $\emptyset$, $X\subseteq Y$, $X\cap Y$, $X\cup Y$, $X\setminus Y$, $X\times Y$, $X^n$, $\pi_i R$, $\mathscr{U}\overline{R}$, and $R[\overline{x}]$.

However, occasionally we avail ourselves of non-absolute definitions. For instance, we occasionally use $\mathbb{V}:=\{x: x=x\}$ as an abbreviation for the universal unary relation. Likewise, we sometimes use $Y^c:=\mathbb{V}\setminus Y$ as an abbreviation for the relative complement of $Y$. But since these are not absolute, we do not add constants and function symbols for these and regard them as mere abbreviations. 

As in other areas, it is convenient to allow constant symbols, both for individuals and for $n$-ary relations. As usual, what constant symbols are available in a given context is specified by the signature. For instance, in a setting where one wants names for individuals, one works in a signature with constants $c,d,\ldots$, and any model $\mathcal{A}$ in this signature comes equipped with interpretations $c^{\mathcal{A}}, d^{\mathcal{A}}, \ldots$ from $A$. Likewise, in a setting where one wants named binary relations for e.g. orders or graphs of functions, one simply works in a signature with binary relation symbols $\preceq, g, \ldots$ on individuals, and any model $\mathcal{A}$ in this signature comes equipped with interpretations $\preceq^{\mathcal{A}}, g^{\mathcal{A}}$ which are subsets of $A^2$. And similarly for $n$-ary relations for values besides $n=2$. Since we always work with at least comprehension for formulas with no second-order quantifiers, the interpretation $C^{\mathcal{A}}$ of an $n$-ary relation symbol~$C$ in a Henkin model $\mathcal{A}$ is identical to an element of $S_n[A]$, and in particular $C^{\mathcal{A}}$ is identical to the set of $n$-tuples $\overline{x}$ from $A^n$ which satisfy the $n$-ary relation $C$ on $\mathcal{A}$. Hence even though we allow extra-logical relation symbols beyond membership, all atomics are equivalent to those formed from membership and identity.

As a deductive theory, \emph{second-order logic} $\sol$ just consists of the comprehension schema~(\ref{eqn:compschema}) and the extensionality axioms, one for each $n\geq 1$, where $R,S$ are $n$-ary:
\begin{equation}
\forall \; R \; \forall \; S \; \bigg(R=S\leftrightarrow \forall \; \overline{x} \; (R\overline{x}\leftrightarrow S\overline{x})\bigg)
\end{equation}
together with the explicit definitions of the minimal set-theoretic apparatus described above. One generates subsystems of second-order logic by keeping extensionality and restricting the comprehension schema. A \emph{$\Sigma^1_1$-formula} is one of the form $\exists \; R_1, \ldots, R_{\ell} \; \varphi$, where $\varphi$ does not contain any second-order quantifiers. A \emph{$\Pi^1_1$-formula} is one of the form $\forall \; R_1, \ldots, R_{\ell}  \; \varphi$, where $\varphi$ does not contain any second-order quantifiers. A formula $\varphi$ of second-order logic is \emph{first-order} if it does not contain any second-order quantifiers. Since the minimal set-theoretic apparatus is all first-order definable, its inclusion does not impact what is $\Sigma^1_1$-definable, $\Pi^1_1$-definable, and first-order definable.

A traditional predicative subsystem is \emph{$\Delta^1_1$-comprehension}, abbreviated $\dcomp$, which consists of the extensionality axioms and the following \emph{$\Delta^1_1$-comprehension schema}:
\begin{equation}
\forall \; \overline{x} \; (\varphi(\overline{x}) \leftrightarrow \psi(\overline{x}))\rightarrow \exists \; R \; \forall \; \overline{x} \; (R\overline{x} \leftrightarrow \varphi(\overline{x}))
\end{equation}
wherein $\varphi$ is $\Sigma^1_1$ and $\psi$ is $\Pi^1_1$. Another traditional predicative system is \emph{$\Sigma^1_1$-choice}, abbreviated $\schoice$ which consists of the extensionality axioms, the comprehension schema (\ref{eqn:compschema}) restricted to first-order formulas, and the following \emph{$\Sigma^1_1$-choice schema}:
\begin{equation}\label{eqn:choice}
\forall \; \overline{x} \; \exists \; R \; \varphi(\overline{x}, R)\rightarrow \exists \; R^{\prime} \; \forall \; \overline{x} \; \varphi(\overline{x}, R^{\prime}[\overline{x}])
\end{equation}
wherein $\varphi(\overline{x}, R)$ is $\Sigma^1_1$ and $R^{\prime}$ is $(n+m)$-ary and $\overline{x}$ is $n$-ary and $R$ is $m$-ary. Finally, the system $\aca$ is simply the extensionality axioms and all instances of the comprehension schema~(\ref{eqn:compschema}) such that the formula $\varphi$ does not contain second-order quantifiers. In addition to these three predicative theories, a traditional intermediary impredicative system is $\scomp$, which consists of the extensionality axioms and the comprehension schema (\ref{eqn:compschema}) for $\Sigma^1_1$-formulas. The names of these subsystems of $\sol$ are taken from the corresponding well-studied theories of second-order arithmetic (\cite{Simpson2009-ra}), but we have added a minus superscript to them to disambiguate and to indicate that they are logical rather than arithmetical in nature.

In any of these systems of second-order logic, one can readily regiment statements about cardinality. Whereas in ordinary set-theoretic mathematics, the cardinality operator $\left|\cdot\right|$ is function from sets to cardinals, in second-order logic one defines the comparative notions $\left|X\right|\leq \left|Y\right|$ and $\left|X\right|< \left|Y\right|$ explicitly as follows:

\begin{defn}\label{defn:cardinality} (Cardinality notions in second-order logic).

One defines $\left|X\right|\leq \left|Y\right|$ if there is an injection $f:X\rightarrow Y$. 

One defines $\left|X\right|< \left|Y\right|$ if both $\left|X\right|\leq \left|Y\right| $ and not $\left|Y\right|\leq \left|X\right|$.

One says that $X$ is \emph{Dedekind infinite} if there is a function $f:X\rightarrow X$ which is injective but not surjective.

One says that $X$ is \emph{Dedekind finite} if $X$ is not Dedekind infinite.

\end{defn}

\noindent Of course, functions $f:X\rightarrow Y$ are identified with their graph as a subrelation of $X\times Y$. Being able to define such notions is one of the chief advantages of working with polyadic as opposed to monadic second-order logic: namely, quantification over binary relations allows one to quantify over functions between different unary second-order entities in addition to being able to quantify over the second-order entities themselves. That said, there are other perspectives from which polyadic second-order quantification is a disadvantage. For instance, plurals are thought to be monadic and not polyadic,\footnote{\cite[117 ff]{Florio2021-el}.} and there are many decidability methods for monadic second-order theories which have no polyadic analogue.\footnote{\cite[Chapter 13]{Barwise1985-ys}.}

\section{Generalized quantifiers: definition and examples}\label{sec:genq}

Second-order logic is an extremely expressive system, even when comprehension is restricted. It is thus a good framework to study generalize quantifiers, since the truth-conditions for these quantifiers are often given by a set-theoretic formula which can be readily transcribed into the object-language of second-order logic. Thus, in this paper, we define:

\begin{defn}\label{defn:gq} (Generalized quantifier).
A \emph{generalized quantifier} is a formula $Q(R_1, \ldots, R_{\ell})$ of second-order logic whose only free variables are the displayed second-order variables $R_1, \ldots R_{\ell}$ with respective arities $n_1, \ldots, n_{\ell}$. 
\end{defn}
\noindent This is abbreviated by saying that the generalized quantifier $Q(R_1, \ldots, R_{\ell})$ is of type $( n_1, \ldots, n_{\ell})$. If the particular values of $n_1, \ldots, n_{\ell}$ are not important, such as when one is just stating general claims about all generalized quantifiers, one further abbreviates this by saying that $Q(\overline{R})$ is of type $( \overline{n})$, or for simplicity that $Q(\overline{R})$ is of type $\overline{n}$. Note that the type $\overline{n}$ specifies both the length of the sequence  $\overline{R}$ as well as the arities of its individual entries. In a Henkin model $\mathcal{A}$, one correspondingly works with elements $(R_1, \ldots, R_{\ell})$ of $S_{n_1}[A]\times \cdots \times S_{n_{\ell}}[A]$, and we similarly just abbreviate this as the element $\overline{R}$ of $S_{\,\overline{n}\,}[A]$.

The tradition of generalized quantifiers typically works with standard models. If $Q(R_1, \ldots, R_{\ell})$ is a generalized quantifier in the sense of Definition~\ref{defn:gq} and $\mathcal{M}$ is a standard model, then the interpretation of $Q(R_1, \ldots, R_{\ell})$ on $\mathcal{M}$ is simply given by a subset of $P(M^{n_1})\times \cdots \times P(M^{n_{\ell}})$. Indeed, this is the traditional understanding of the semantic value of a generalized quantifier within linguistics, namely, it is a collection of tuples of subsets of Cartesian products of the underlying domain. In restricting to collections which are defined by second-order formulas, Definition~\ref{defn:gq} is more restrictive than this tradition; although, as the subsequent list of examples suggests, nothing is lost by this restriction. In being applicable to Henkin models which are not standard models, Definition~\ref{defn:gq} is more expansive than this tradition. It is a chief burden of this paper to make a case that this is a fruitful expansion.

Here are some famous examples of generalized quantifiers of type $(1,1)$, where we list their truth-conditions followed by an example:
\begin{ex}\label{ex:all}
$\mathsf{All}(G,H)$ iff $\forall \; x \; (Gx\rightarrow Hx)$. All Geography majors are happy.
\end{ex}
\begin{ex}\label{ex:some}
$\mathsf{Some}(G,H)$ iff $\exists \; x \; (Gx\wedge Hx)$. Some Geography majors are happy.
\end{ex}
\noindent Of course, in the tradition of Aristotelian logic, one defines the universal so that it has existential import:\footnote{\cite[24]{Peters2008aa}. See the conversions in the \emph{Prior Analytics} I.2.25$^a$10 (\cite{Aristotle2014-al}).}
\begin{ex}\label{ex:allar}
$\mathsf{All}^{\ast}(G,H)$ iff $\forall \; x \; (Gx\rightarrow Hx) \wedge \exists \; x \; Gx$. All Geography majors are happy.
\end{ex}
\noindent Other familiar type $(1,1)$ generalized quantifiers are the following:\footnote{\cite[121, 62]{Peters2008aa}. For discussion of why `two' means `exactly two' rather than `at least two', see \cite[\S{9.2}]{Szabolcsi2010-vi}. For discussion of how `most' can be understood as a superlative, see \cite{Hackl2009-mj}, \cite[\S{5.6}]{Szabolcsi2010-vi}.}
\begin{ex}\label{ex:two}
$\mathsf{Two}(F,G)$ iff $\exists \; y_1 \; \exists \; y_2 \; Fy_1\wedge Fy_2\wedge y_1\neq y_2 \wedge Gy_1\wedge Gy_2 \wedge \forall \; z \; ((Fz\wedge Gz)\rightarrow (z=y_1 \vee z=y_2))$. Two French majors are Geography majors.
\end{ex}
\begin{restatable}{ex}{exmost}\label{ex:most}
$\mathsf{Most}(F,G)$ iff $\left|F\cap G\right|>\left|F\setminus G\right|$. Most French majors go to study abroad.
\end{restatable}
\noindent In this, we use the regimentation of $\left|X\right|>\left|Y\right|$ from Definition~\ref{defn:cardinality}.

But there are many examples of natural quantifiers of different type than $(1,1)$. Here is an example of a type (1) quantifier associated to Montagovian individuals:\footnote{\cite[93-94]{Peters2008aa}. See \cite[806-807]{Glanzberg2006-is} for a balanced discussion of the importance of including this example.\label{fn:john}}

\begin{ex}\label{ex:montagueindivdiaul}
$\mathsf{Christopher}(G)$ iff $Gc$, where $c$ is an individual constant symbol for the individual Christopher. If most French majors go to study abroad then Christopher goes to study abroad.
\end{ex}
\noindent Here is an example of type $(1,1,1)$:\footnote{\cite[154]{Peters2008aa}.}
\begin{ex}\label{ex:more}
$\mathsf{More\mbox{-}than}(F,G,H)$ iff $\left|F\cap H\right|>\left|G\cap H\right|$. More French majors than German majors are honors students.
\end{ex}
\noindent Here is an example of type $(2,2)$:\footnote{\cite[352]{Peters2008aa}, \cite[1080]{Keenan2011-pt}. See \cite[Chapter 3]{de-Swart1993-sy} for a general study of adverbs using ideas from generalized quantifiers.}
\begin{ex}\label{ex:usually}
$\mathsf{Usually}(R,R^{\prime})$ iff $\left|R\cap R^{\prime}\right|>\left|R\setminus R^{\prime}\right|$. Cats usually like dogs, which is regimented as: there are more ordered pairs $R$ of cats and dogs that stand in the like relation $R^{\prime}$ than there are ordered pairs of $R$ of cats and dogs that stand in the dislike relation~$\neg R^{\prime}$.
\end{ex}

Lastly, consider the following generalized quantifier of type $(1,1,2)$:\footnote{\cite[60 equation (19), 63 equation (t)]{Barwise1979-ns}, \cite[71, 363]{Peters2008aa}, \cite[110]{Szymanik2016-ue}.} 
\begin{ex}\label{ex:branchmost}
$\mathsf{Branching\mbox{-}most}(F,G,R)$ iff $\exists \; F_0\subseteq F \; \exists \; G_0\subseteq G \; \mathsf{Most}(F,F_0) \wedge \mathsf{Most}(G,G_0) \wedge F_0\times G_0\subseteq R$. Most French majors and most Geography majors all respect one another.
\end{ex}
\noindent This is Barwise's template for Hintikka's branching quantifiers, where the generalization emerges by replacing one or both instances of $\mathsf{Most}$ by other type $(1,1)$ generalized quantifiers and perhaps adding more than two of them or just one of them (cf. (\ref{eqn:rebarwisebranching})). Barwise's example is intended to have the stated truth conditions, and hence it is distinct from $\mathsf{Usually}(F\times G,R)$. However, it seems plausible that the truth conditions of Barwise's example derives from the reciprocal ``one another'' or  the ``all'' in the verb phrase,\footnote{\cite[78 ff]{Hoeksema1983-jo}, \cite[165]{Steedman2011-co}, \cite[432]{May1989-ag}, \cite[258]{Liu1996-ma}. But it has been argued that the reciprocal itself has the semantics of a generalized quantifier (cf. \cite[Chapter 8]{Szymanik2016-ue}).} and hence for reasons ultimately extrinsic to the generalized quantifier in the noun phrase. There is then some reason to regard Example~\ref{ex:branchmost} as slightly set apart from the previous examples.\footnote{In this connection, it should be noted that branching quantification is mentioned only in passing, or not at all, in many recent surveys (\cite{Glanzberg2006-is}, \cite{Szabolcsi2010-vi}, \cite{Keenan2011-pt}, \cite{Keenan2012-nv}, \cite{Partee2012-on}, \cite{Sternefeld2020-vf}).}

\section{Substructures and Henkin domain independence}\label{sec:henkinconstraints:1}

In formal semantics, one of the main conditions on natural language quantifiers is the \emph{domain independence condition},\footnote{For the original notion on standard models, see \cite[855-856]{Keenan1997-vc}, \cite[250, Definition 17]{Gamut1991-gc}, \cite[279]{Westerstahl2007-fb}, \cite[101, 108]{Peters2008aa}, \cite[\S{4.2.2.1}]{Szabolcsi2010-vi}, \cite[1067]{Keenan2011-pt}. The concept is originally due to van Benthem, see \cite[1067]{Keenan2011-pt} and  \cite[8]{van-Benthem1986-qu}.} which says that the quantifier is both upwards and downwards invariant under expansions of the first-order domain. The natural notion of \emph{expansion} in Henkin models is the model-theoretic notion of substructure:

\begin{defn}\label{defn:substructure} (Substructure).

Suppose that $\mathcal{A}, \mathcal{B}$ are two Henkin models. Then $\mathcal{A}$ is a \emph{substructure} of $\mathcal{B}$ if $A\subseteq B$ and $S_n[A]\subseteq S_n[B]$ for all $n\geq 1$, and if for all individual constant symbols $c$ in the signature, $c^{\mathcal{B}}=c^{\mathcal{A}}$ and for all $n$-ary relation symbols $C$ in the signature, $C^{\mathcal{B}}\cap A^n = C^{\mathcal{A}}$.
\end{defn}
\noindent If the two models are standard and one is working in signature with no extra-logical constants, then this happens iff $A\subseteq B$. But note that if $\mathcal{A}$ is merely a Henkin model, then specifying an extension requires not only specifying a superset $B$ of $A$ but also an associated sequence of supersets $S_n[B]$ of $S_n[A]$ which satisfy $S_n[B]\subseteq P(B^n)$.
\begin{defn} (The standard superstructure).

Suppose $\mathcal{A}$ is a Henkin model, which may or may not be standard. Let $\mathcal{B}$ be the Henkin model given by $B=A$ and $S_n[B]=P(A^n)$ for all $n\geq 1$ and by interpreting the constant symbols the same. We call $\mathcal{B}$ the \emph{standard superstructure} of $\mathcal{A}$.
\end{defn}
\noindent In this case, the superstructure is formed by keeping the first-order domain the same and by expanding the domains associated to the higher-order objects by recourse to a metatheoretic operation. Another kind of example is where we shrink the first-order domain and correspondingly take the second-order objects to be their intersection with the new first-order domain:
\begin{defn}\label{defn:internal} (Internal substructures).

Let $\mathcal{B}$ be a Henkin model. Let $A$ be a non-empty set in $S_1[B]$ which contains $c^{\mathcal{B}}$ for all individual constant symbols $c$. Then the \emph{internal substructure $\mathcal{A}$ of $\mathcal{B}$ associated to $A$} is the Henkin model with first-order domain $A$ and with $S_n[A]=S_n[B]\cap P(A^n)$ for all $n\geq 1$, and with $c^{\mathcal{A}}=c^{\mathcal{B}}$ for all individual constant symbols $c$ in the signature and with $C^{\mathcal{A}}=C^{\mathcal{B}}\cap A^n$ for all $n$-ary relation symbols $C$ in the signature. 
\end{defn}
Note that in this case, if $\mathcal{B}$ is standard then $\mathcal{A}$ is standard. 

For some concrete examples of pairs of Henkin models $\mathcal{A},\mathcal{B}$ where $\mathcal{A}$ is a substructure of $\mathcal{B}$ and $A=B$ while $S_n[A]\subsetneq S_n[B]$ for $n\geq 1$, see Examples~\ref{ex:disjunction:2}-\ref{ex:disjunction:3} in~\S\ref{sec:disjunctiveexamples}. One can also produce such examples by using the Downward L\"owenheim-Skolem Theorem.\footnote{Followed by an application of Proposition~\ref{prop:iso}.}

We draw attention to a feature of many theories which is related to internal substructures:
\begin{defn}\label{defn:closeddownardinternal} (Closed downwards under internal substructures).

A theory $T$ is \emph{closed downward under internal substructures} if for all Henkin models $\mathcal{B}$ of $T$ and all internal substructures $\mathcal{A}$ of $\mathcal{B}$, one has that $\mathcal{A}$ is a model of $T$ as well. 
\end{defn}
\noindent Many logical theories, including $\mathsf{ACA}_0^{-}$, $\mathsf{\Delta^1_1\mbox{-}CA}_0^{-}$, $\schoice$, $\scomp$ and $\sol$, are closed downward under internal substructures. This is because the formula $\varphi$ in the Comprehension Schema~(\ref{eqn:compschema}) is allowed to use $A$ as a parameter and hence one can bind all the quantifiers in $\varphi$ to $A$ without increasing the relevant kinds of complexity (cf. Definition~\ref{defn:binding} for the definition of binding).

We now define the generalization of domain independence to Henkin models:
\begin{defn}\label{defn:HDI} (Henkin independence notions).

Suppose that $T$ is a theory and $Q(\overline{R})$ is a generalized quantifier. Then we say that $Q$ is \emph{Henkin upwards independent  relative to $T$} if for all pairs of Henkin models $\mathcal{A}, \mathcal{B}$ of $T$, if $\mathcal{A}$ is a substructure of $\mathcal{B}$, then for all $\overline{R}$ from $S_{\,\overline{n}\,}[A]$:
\begin{equation}\label{eqn:ext:2}
\mathcal{A}\models Q(\overline{R}) \Longrightarrow  \mathcal{B}\models Q(\overline{R})
\end{equation}
If the arrow is reversed, we say that $Q$ is \emph{Henkin downwards independent relative to $T$}. If the biconditional holds, then we say that $Q$ is \emph{Henkin domain independent relative to $T$}. If $T$ is clear from context, then we omit ``relative to $T$'' for brevity. 
\end{defn}

We look at simple examples which violate Henkin domain independence in \S\ref{sec:disjunctiveexamples}. We discuss the motivations for Henkin domain independence in \S\ref{sec:discuss:2}.

\section{Henkin conservativity}\label{sec:henkinconstraints:2}

In order to define conservativity,\footnote{For the original notion on standard models, see \cite[852-853]{Keenan1997-vc}, \cite[255]{Westerstahl2007-fb}, \cite[138]{Peters2008aa}, \cite[\S{4.2.2.1}]{Szabolcsi2010-vi}, \cite[1066]{Keenan2011-pt} for the case of type $(1,1)$, and see \cite[77]{Keenan1985-cq} for the general case. Conservativity was first defined by Keenan \cite[371, Definition 9]{Keenan1981-iv}; see also \cite[275]{Keenan1986-pq}. See Barwise-Cooper \cite[170, 178-179]{Barwise1981-ur} for the closely related notion of ``living on.''} one has to work with generalized quantifiers $Q(\overline{R},\overline{R}^{\prime})$ where the free variables are split into two non-empty parts $\overline{R},\overline{R}^{\prime}$. To avoid a great deal of notational clutter, we are going to suppose that the tuple $\overline{R}$ comes before the tuple $\overline{R}^{\prime}$ in the enumeration of the free variables in $Q(\overline{R},\overline{R}^{\prime})$, although as one will see that nothing in what follows depends on this.\footnote{Some flexibility as to the placement of the tuple has proven useful in the linguistics literature as well, and is one of the responses to the apparent non-conservativity of ``only.'' See the discussion of weak conservativity in \cite[578]{Zuber2019-lc}.} We define:
\begin{defn}\label{defn:Hcon} (Henkin conservative).

Suppose that $T$ is a theory and $Q(\overline{R},\overline{R}^{\prime})$ is a generalized quantifier.  Then we say that $Q$ is \emph{Henkin conservative relative to $T$} if $T$ proves:
\begin{equation}\label{eqn:henkinconservation}
\forall \; \overline{R} \; \forall \; \overline{R}^{\prime} \; \big(Q(\overline{R}, \overline{R}^{\prime})\leftrightarrow Q(\overline{R}, \mathscr{U}\overline{R}\cap \overline{R}^{\prime})\big)
\end{equation}
If $T$ is clear from context, then we omit ``relative to $T$'' for brevity. 

\end{defn}

In equation~(\ref{eqn:henkinconservation}), the notation $\mathscr{U}\overline{R}\cap \overline{R}^{\prime}$ means the tuple of the same length as $\overline{R}^{\prime} = (R_1^{\prime}, \ldots, R_{\ell^{\prime}}^{\prime})$ whose $i$-th component is $\mathscr{U}\overline{R}\cap R_i^{\prime}$, where we use the convention from (\ref{eqn:conventionpowers}). 

The content of Henkin conservation is best understood by way of example:
\begin{ex}
If $Q(X,Y)$ is of type $( 1,1)$, such as in Examples~\ref{ex:all}-~\ref{ex:most}, then (\ref{eqn:henkinconservation}) is simply the universal closure of $Q(X,Y)\leftrightarrow Q(X,X\cap Y)$.
\end{ex}
\noindent To spell this out with the sentence from Example~\ref{ex:most}, the biconditional $Q(X,Y)\leftrightarrow Q(X,X\cap Y)$ is: most French majors go to study abroad iff most French majors are French majors who go to study abroad. The intuitiveness of biconditionals like this is some of the primary evidence for conservativity.\footnote{\cite[853]{Keenan1997-vc}, \cite[138-139]{Peters2008aa}, \cite[801]{Glanzberg2006-is}.}
\begin{ex}\label{ex:type111conserv}
If $Q(X,Y,Z)$ is of type $( 1,1,1)$ as in Example~\ref{ex:more} with the partition between $X,Y$ and $Z$, then (\ref{eqn:henkinconservation}) is simply the universal closure of $Q(X,Y,Z)\leftrightarrow Q(X, Y, (X\cup Y)\cap Z)$.
\end{ex}
\begin{ex}
If $Q(R,R^{\prime})$ is of type $( 2,2)$ as in Example~\ref{ex:usually}, then (\ref{eqn:henkinconservation}) is simply the universal closure of $Q(R,R^{\prime})\leftrightarrow Q(R,(\mathscr{U}R)^2\cap R^{\prime})$.
\end{ex}

We discuss further the motivations for Henkin conservativity in \S\ref{sec:discuss:2}.

Henkin conservativity has the following consequence for the empty relation:
\begin{rmk}\label{rmk:consequence:henkinempty} (Henkin conservativity and the empty relation).

Suppose that $Q(\overline{R},\overline{R}^{\prime})$ is Henkin conservative relative to $T$. Then $T$ proves all three of the following:
\begin{equation*}
\forall \;\overline{R}^{\prime} \; \big(Q(\emptyset, \overline{R}^{\prime})\leftrightarrow Q(\emptyset,\emptyset)\big), \hspace{5mm}  Q(\emptyset,\emptyset)\leftrightarrow \forall \;\overline{R}^{\prime} \; Q(\emptyset, \overline{R}^{\prime}), \hspace{5mm}  Q(\emptyset,\emptyset)\leftrightarrow \exists \;\overline{R}^{\prime} \; Q(\emptyset, \overline{R}^{\prime})
\end{equation*}
The first of these follows because $\overline{R}$ is empty iff its field $\mathscr{U}\overline{R}$ is empty, and the intersection of anything with the empty relation is the empty relation. The other two follow from the first by elementary quantifier manipulations.
\end{rmk}

\section{The disjunctive examples}\label{sec:disjunctiveexamples}

Below are three examples of quantifiers that are \emph{not} Henkin domain independent but are Henkin conservative. This kind of example is well-known.\footnote{See Keenan \cite[\S{4}, immediately below (37), p. 56]{Keenan1996-yi}; and \cite[152]{de-Swart1993-sy}, \cite[251]{Gamut1991-gc}.} We elaborate on it by producing examples where the disjunctive behavior pertains to both the first- and second-order part.

\begin{ex}\label{ex:disjunction:1} (Disjunctive Example 1).

Let $Q(X,Y)$ be $(\neg \varphi_3 \rightarrow X\cap Y\neq \emptyset)\wedge (\varphi_3 \rightarrow X\subseteq Y)$, where $\varphi_3$ is $\exists \; x \; \exists \; y \; \exists \; z \; x\neq y \wedge y\neq z \wedge x\neq z$. That is, $Q(X,Y)$ is the type $(1,1)$ quantifier which is the existential quantifier if there are fewer than three elements in the domain, and which is the universal quantifier if there are at least three elements in the domain.

Consider a two element standard model $\mathcal{A}$ and a three element standard model~$\mathcal{B}$ such that $\mathcal{A}$ is a substructure of $\mathcal{B}$. 

Choose $X,Y$ in $S_1[A]$ such that $X\cap Y\neq \emptyset$ but not $X\subseteq Y$. Then $\mathcal{A}\models Q(X,Y)$ but $\mathcal{B}\models \neg Q(X,Y)$. Hence $Q(X,Y)$ is not upwards Henkin independent relative to any theory $T$ which is true on all finite standard models.

Choose $Z$ in $S_1[A]$. Then $\mathcal{B}\models Q(\emptyset,Z)$ but $\mathcal{A}\models \neg Q(\emptyset,Z)$. Hence $Q(X,Y)$ is not downwards Henkin independent relative to any theory $T$ which is true on all finite standard models.

But $Q(X,Y)$  is Henkin conservative relative to any theory extending $\mathsf{ACA}_0^{-}$ since both the existential quantifier and universal quantifier are Henkin conservative.
\end{ex}

\begin{ex}\label{ex:disjunction:2} (Disjunctive Example 2).

Let $Q(X,Y)$ be $(\varphi_{fin} \rightarrow X\cap Y\neq \emptyset)\wedge (\neg \varphi_{fin} \rightarrow X\subseteq Y)$, where $\varphi_{fin}$ says that the domain is Dedekind finite. That is, $Q(X,Y)$ is the type $(1,1)$ quantifier which is the existential quantifier if the domain is Dedekind finite, and which is the universal quantifier if the domain is Dedekind infinite (cf. Definition~\ref{defn:cardinality} for the definition of Dedekind finitude and infinitude).

Consider the Henkin model $\mathcal{A}$ such that $A$ is the the complex numbers $\mathbb{C}$ with constant symbols for the field structure and with $S_n[A]$ being the first-order definable subsets of the $n$-th Cartesian product $\mathbb{C}^n$ of the complex field. Consider the standard model $\mathcal{B}$ such that $B$ is the the complex numbers $\mathbb{C}$ with constant symbols for the field structure. Then $\mathcal{A}$ is a model of $\mathsf{\Delta^1_1\mbox{-}CA}_0^{-} + \varphi_{fin}$,\footnote{It is a model of $\mathsf{\Delta^1_1\mbox{-}CA}_0^{-}$ since the complex field has theory which is uncountably categorical and hence saturated (cf. \cite[Theorem 5.2.11]{Tent2012-jw}, \cite[Theorem VI.1.12]{Hungerford1974-wx}) and hence recursively saturated, and the Barwise-Schlipf method shows that the second-order models associated to recursively saturated structures satisfy $\mathsf{\Delta^1_1\mbox{-}CA}_0^{-}$ (\cite[Theorem 2.2]{Barwise1975-xg}, \cite[Lemma IX.4.3]{Simpson2009-ra}). It is a model of $\varphi_{fin}$ by the famous Ax–Grothendieck Theorem (cf. \cite[46]{Tent2012-jw}). \label{fnax}} while $\mathcal{B}$ is a model of $\mathsf{\Delta^1_1\mbox{-}CA}_0^{-}+\neg \varphi_{fin}$.

At a certain level of idealization, $\mathcal{A}$ is the setting of algebraic geometry, and $\mathcal{B}$ is the setting of complex analysis. Hence, this is very much a real-life example of two adjacent disciplines which differ from one another not in terms of their first-order objects but in terms of the second-order entities which they countenance.

Choose $X,Y$ in $S_1[A]$ such that $X\cap Y\neq \emptyset$ but not $X\subseteq Y$. Then $\mathcal{A}\models Q(X,Y)$ but $\mathcal{B}\models \neg Q(X,Y)$. Hence $Q(X,Y)$ is not upwards Henkin independent relative to $\mathsf{\Delta^1_1\mbox{-}CA}_0^{-}$. 

Choose $Z$ in $S_1[A]$. Then $\mathcal{B}\models Q(\emptyset,Z)$ but $\mathcal{A}\models \neg Q(\emptyset,Z)$. Hence $Q(X,Y)$ is not downwards Henkin domain independent relative to $\mathsf{\Delta^1_1\mbox{-}CA}_0^{-}$. 

Further, $Q(X,Y)$  is Henkin conservative relative to any theory extending $\mathsf{ACA}_0^{-}$ since both the existential quantifier and universal quantifier are Henkin conservative.
\end{ex}

\begin{ex}\label{ex:disjunction:3} (Disjunctive Example 3).

Let $\varphi_{PA^2}$ say that there is a model of second-order Peano arithmetic:
\begin{align*}
 \exists \; z \; & \exists \; N \; \exists \; S \; \\
\&\; &\forall \; x \; (Nx \rightarrow \exists \; ! \; y \; (Ny \wedge Sxy)) \\
\&\; & \forall \; x \; \forall \; y \; \forall \; w \; \big((Nx\wedge Ny\wedge Nw \wedge Sxw\wedge Syw)\rightarrow x=y\big) \\
\&\; & \forall \; x \; (Nx\rightarrow \neg Sxz) \\
\&\; & \forall \; X \; \bigg(\big(X\subseteq N \; \& \; Xz \; \& \; \big(\forall \; x \; \forall \; y \; ((Nx \; \& \; Ny \; \& \; Sxy \; \& \; Xx)\rightarrow Xy)\big)\big)\rightarrow N\subseteq X\bigg)
\end{align*}
Let $Q(X,Y)$ be $(\neg \varphi_{PA^2} \rightarrow X\cap Y\neq \emptyset)\wedge ( \varphi_{PA^2} \rightarrow X\subseteq Y)$. That is, $Q(X,Y)$ is the type $(1,1)$ quantifier which is the existential quantifier if the domain does not contain a model of second-order Peano arithmetic, and which is the universal quantifier if the domain does contain a model of second-order Peano arithmetic.

Consider the Henkin model $\mathcal{A}$ such that $A$ is uncountable and $S_n[A]$ is the collection of first-order definable subsets of $A$ as a first-order structure in empty signature. Consider the standard model $\mathcal{B}$ such that $B$ is $A$. Then $\mathcal{A}$ is a model of $\mathsf{\Delta^1_1\mbox{-}CA}_0^{-} + \neg \varphi_{PA^2}$,\footnote{It is a model of $\mathsf{\Delta^1_1\mbox{-}CA}_0^{-}$ since $A$ as a first-order structure in the empty theory is uncountably categorical and hence saturated (cf. \cite[Theorem 5.2.11]{Tent2012-jw}) and hence recursively saturated, and the Barwise-Schlipf method shows that the second-order models associated to recursively saturated structures satisfy $\mathsf{\Delta^1_1\mbox{-}CA}_0^{-}$ (\cite[Theorem 2.2]{Barwise1975-xg}, \cite[Lemma IX.4.3]{Simpson2009-ra}). It is a model of $\neg \varphi_{PA^2}$ since the sets in $S_1[A]$ are finite or cofinite. But e.g. the evens in a model of Peano arithmetic are infinite and coinfinite.} while $\mathcal{B}$ is a model of $\mathsf{\Delta^1_1\mbox{-}CA}_0^{-}+\varphi_{PA^2}$.

The model $\mathcal{A}$ is the setting where the only properties are finite and cofinite collections of objects, and $\mathcal{B}$ is the setting where a richer array of properties exist.

As in the previous example, $Q(X,Y)$ is neither upwards nor downwards Henkin independent relative to $\mathsf{\Delta^1_1\mbox{-}CA}_0^{-}$, but it is Henkin conservative relative to any theory extending $\mathsf{ACA}_0^{-}$.
\end{ex}

\section{Definability, bounded definability, and absoluteness}\label{sec:defnandboundeddefn}

In this short section, we say what it means for a generalized quantifier itself to be definable and bounded definable, and we discuss the absoluteness of certain definitions.

\begin{defn}\label{defn:definable} (Definability relative to $T$).

Suppose that $T$ is a theory and $Q(\overline{R})$ is a generalized quantifier. Then $Q$ is \emph{$\Sigma^1_1$-definable relative to $T$} if there is a $\Sigma^1_1$-formula $\Phi(\overline{R})$ with exactly the free variables displayed such that $T$ proves
\begin{equation}
\forall \; \overline{R} \; (Q(\overline{R})\leftrightarrow \Phi(\overline{R}))
\end{equation}
Further,  \emph{$\Pi^1_1$-definable relative to $T$} is defined exactly the same, but with the $\Sigma^1_1$-formula replaced by a $\Pi^1_1$-formula. And \emph{first-order definable relative to $T$} is defined exactly the same but with $\Sigma^1_1$-formula replaced by a first-order formula. Finally $Q$ is \emph{$\Delta^1_1$-definable relative to $T$} if it is both $\Sigma^1_1$- and $\Pi^1_1$-definable relative to~$T$.  
If $T$ is clear from context, then we omit ``relative to $T$'' for brevity.
\end{defn}

We need a similar notion of bounded quantifiers and bounded definability:
\begin{defn}\label{defn:binding} (Bounded definability notions).
\begin{enumerate}[leftmargin=*]
    \item\label{defn:binding1} If $X$ is a unary second-order term, then a formula $\varphi$ has \emph{all quantifiers bounded to $X$} if every first-order quantifier in $\varphi$ has the form $\exists \; x \; Xx\wedge \cdots$ or the form $\forall \; x \;(Xx\rightarrow \cdots)$, and if every second-order quantifier in $\varphi$ has the form $\exists \; R \; R\subseteq X^n \wedge \cdots$ or the form $\forall \; R \; (R\subseteq X^n \rightarrow \cdots$, where $R$ is $n$-ary.
    \item\label{defn:binding2} A formula is \emph{bound} if there is an unary second-order term $X$ such that all quantifiers in the formula are bound to $X$ and the free variables of $X$ do not occur bound in the formula.
    \item\label{defn:binding3} If $\varphi$ is a formula and $X$ is a unary term whose free variables do not occur bound in $\varphi$, then we define $\varphi^X$ to be the result of first binding the first-order quantifiers in $\varphi$ to $X$ and subsequently binding the second-order quantifiers in $\varphi$ to $X$.  
\end{enumerate}
\end{defn}
Obviously, $\varphi^X$ as in (\ref{defn:binding3}) is bound in the sense of (\ref{defn:binding2}); and any formula bound in the sense of (\ref{defn:binding2}) is of the form $\varphi^X$.

Here are three examples:
\begin{ex}
If $R$ is a binary second-order variable and $X$ is unary second-order variable, then $(\exists \; R \; \forall \; x \; Rxx)^X$ is $\exists \; R \; R\subseteq X^2 \wedge \forall \; x \; (Xx\rightarrow Rxx)$. This last formula is bound, with all quantifiers bound to $X$.
\end{ex}
\begin{ex}
If $R$ is a binary second-order variable and $R^{\prime}$ is a binary second-order variable and $X$ is a unary second-order variable, then $\mathscr{U}R$ is a unary second-order term and $(\exists \; X \; \forall \; x \; (Xx\rightarrow R^{\prime}xx))^{\mathscr{U}R}$ is  $\exists \; X\subseteq \mathscr{U}R \; \forall \; x \; ((\mathscr{U}R)x \rightarrow (Xx\rightarrow R^{\prime}xx))$. This last formula is bound, with all quantifiers bound to the field~$\mathscr{U}R$ of~$R$. 
\end{ex}
\begin{ex}
The formulas $\left|X\right|\leq \left|Y\right|$ and $\left|X\right|<\left|Y\right|$ from Definition~\ref{defn:cardinality} are bound to $X\cup Y$. Likewise, the formula for $X$ being Dedekind infinite and the formula for $X$ being Dedekind finite are bound to $X$.
\end{ex}

However, we note the following, which is the familiar problem of predicate logic on empty domains:\footnote{It might be worth exploring ways to modify the logic to handle empty domains, along the lines of \cite{Hailperin1953-xk}, \cite{Fitting1971-wk}. However, given the other tools used in this paper, that would require a deeper understanding of the metatheory of such logics than I currently have.}$^{\mbox{,\;}}$\footnote{If one is just interested in model theory, one can just accept empty domains and the consequences spelled out in the remark and let the chips fall where they may-- this is how Hodges proceeds \cite[p. 2, p. 46, Exercise 7]{Hodges1993-ux}. However, without corresponding changes to the deductive system (along the lines of the previous footnote), this will cause problems when one starts doing metatheory related to the Soundness and Completeness Theorems, which I am very much using throughout the paper.}

\begin{rmk}\label{rmk:binding:poorly} (Binding interacts poorly with the empty set).

Suppose that $\mathcal{A}$ is a model of $\mathsf{ACA}_0^-$ and $\overline{x}$ is from $A$ and $\overline{R}$ is from $S_{\overline{n}}[A]$ and~$X$ is term with free variables among $\overline{x},\overline{R}$, and suppose that $X$ is empty, that is, $X=\emptyset$. 

If $\varphi(\overline{x},\overline{R})$ begins with a first-order existential quantifier, then $\mathcal{A}\models \neg \varphi^X(\overline{x},\overline{R})$.

If $\varphi(\overline{x},\overline{R})$ begins with a first-order universal quantifier, then $\mathcal{A}\models \varphi^X(\overline{x},\overline{R})$.
\end{rmk}

Despite this poor behavior in this limiting case, binding notions are important for their absoluteness properties. We encountered such properties already in Proposition~\ref{prop:absolutenboolean}. It is a recurring theme in mathematical logic: for instance, bounded formulas are absolute between transitive models of set theory, and recursive formulas are absolute between models of first-order Peano arithmetic.\footnote{\cite[Lemma 12.9]{Jech2007-bp}, \cite[Lemma IV.1.20]{Hajek1998-rs}.} The following is the version of absoluteness which is useful in the setting of second-order logic:

\begin{restatable}{thm}{thmstable}\label{thm:stable} (Grades of absoluteness and bounded definability).

Suppose that $\mathcal{A},\mathcal{B}$ are two Henkin models of $\mathsf{ACA}_0^{-}$ with $\mathcal{A}$ a substructure of~$\mathcal{B}$. Suppose that $\varphi(\overline{x},\overline{R})$ is a formula with free variables among $\overline{x},\overline{R}$, and suppose $X$ is a unary term. Suppose that $\overline{x},\overline{R}$ are from $\mathcal{A}$ and the free variables of $X$ are from~$\mathcal{A}$ and are not bound in $\varphi(\overline{x},\overline{R})$.

If $\varphi(\overline{x},\overline{R})$ is $\Sigma^1_1$, then
\begin{equation}\label{eqn:stable:conditional}
\mathcal{A}\models \varphi^X(\overline{x},\overline{R}) \Longrightarrow \mathcal{B}\models \varphi^X(\overline{x},\overline{R})
\end{equation}
If $\varphi(\overline{x},\overline{R})$ is $\Pi^1_1$ then the arrow is reversed. If $\varphi(\overline{x},\overline{R})$ is equivalent to a single $\Sigma^1_1$-formula on both $\mathcal{A},\mathcal{B}$ and a single $\Pi^1_1$-formula on both $\mathcal{A},\mathcal{B}$, then the biconditional holds. In particular, if $\varphi(\overline{x},\overline{R})$ is first-order, then the biconditional holds.
\end{restatable}

We prove this in \S\ref{sec:stable}, and it builds off of Proposition~\ref{prop:absolutenboolean}. Note that Example~\ref{ex:disjunction:1} shows that one cannot remove the binding to $X$ and maintain the conditional in (\ref{eqn:stable:conditional}) for first-order formulas $\varphi(\overline{x},\overline{R})$.

In the case of internal substructures (cf. Definition~\ref{defn:internal}) one can improve Theorem~\ref{thm:stable} as follows:
\begin{restatable}{thm}{thminternal}\label{thm:internal} (More extensive absoluteness on internal substructures).

If $\mathcal{A}$ is an internal substructure of $\mathcal{B}$, then one has $\mathcal{A}\models \varphi(\overline{x},\overline{R})$ iff $\mathcal{B}\models \varphi^A(\overline{x},\overline{R})$, for all formulas $\varphi$ and $\overline{x}$ from $A$ and $\overline{R}$ from $S_{\,\overline{n}\,}[A]$.
\end{restatable}
\noindent We prove this in \S\ref{sec:stable}. It has the following immediate corollary:
\begin{cor}\label{cor:stable:standard:internal}  (More extensive absoluteness on standard models).

If $\mathcal{A},\mathcal{B}$ standard and $\mathcal{A}$ is a substructure of $\mathcal{B}$, then one has $\mathcal{A}\models \varphi(\overline{x},\overline{R})$ iff $\mathcal{B}\models \varphi^A(\overline{x},\overline{R})$, for all formulas $\varphi$ and $\overline{x}$ from $A$ and $\overline{R}$ from $S_{\,\overline{n}\,}[A]$.
\end{cor}

Having defined bounded quantifiers, we can now define bounded definability:\footnote{It is possible to define a less restrictive version of this where we bind to $\mathscr{U} \overline{R}\cup \mathscr{U}\overline{R^{\prime}}$ rather than to $\mathscr{U} \overline{R}$. However, the definition given anticipates the application to Henkin conservation in Theorem~\ref{thm:fefermanesqfour}.}
\begin{defn}\label{defn:bounded} (Bounded definability of generalized quantifiers).

Suppose that $T$ is a theory and $Q(\overline{R},\overline{R}^{\prime})$ is a generalized quantifier. Then $Q$ is \emph{bounded $\Sigma^1_1$-definable relative to $T$} if there is a $\Sigma^1_1$-formula $\Phi(\overline{R}, \overline{R}^{\prime})$ with all free variables displayed such that $T$ proves
\begin{equation}\label{eqn:bndsigma11}
\forall \; \overline{R} \;\forall \; \overline{R^{\prime}} \; (Q(\overline{R}, \overline{R^{\prime}})\leftrightarrow \Phi^{\mathscr{U}\overline{R}}(\overline{R},\overline{R^{\prime}}))
\end{equation}
Further,  \emph{bounded $\Pi^1_1$-definable relative to $T$} is defined exactly the same, but with the $\Sigma^1_1$-formula replaced by a $\Pi^1_1$-formula. And \emph{bounded first-order definable relative to $T$} is defined exactly the same but with $\Sigma^1_1$-formula replaced by a first-order formula. Moreover, $Q$ is \emph{bounded $\Delta^1_1$-definable relative to $T$} if it is both bounded $\Sigma^1_1$-definable relative to $T$ and bounded-$\Pi^1_1$-definable relative to~$T$. Finally, $Q$ is \emph{bounded definable relative to $T$} if $T$ proves (\ref{eqn:bndsigma11}) for some formula $\Phi(\overline{R},\overline{R}^{\prime})$.
If $T$ is clear from context, then we omit ``relative to $T$'' for brevity.
\end{defn}

One then has the following immediate Corollary:
\begin{cor}\label{cor:stable} (Grades of bounded definability suffices for grades of domain invariance).

Suppose that $T$ is a theory and $Q(\overline{R},\overline{R}^{\prime})$ is a generalized quantifier.
\begin{enumerate}[leftmargin=*]
    \item\label{cor:stable:1} If $Q(\overline{R},\overline{R}^{\prime})$ is bounded $\Sigma^1_1$-definable relative to $T$, then $Q(\overline{R},\overline{R}^{\prime})$ is upward Henkin invariant relative to $T$.
    \item\label{cor:stable:2} If $Q(\overline{R},\overline{R}^{\prime})$ is bounded $\Pi^1_1$-definable relative to $T$, then $Q(\overline{R},\overline{R}^{\prime})$ is downard Henkin invariant relative to $T$.
    \item\label{cor:stable:3} If $Q(\overline{R},\overline{R}^{\prime})$ is bounded $\Delta^1_1$-definable relative to $T$, then $Q(\overline{R},\overline{R}^{\prime})$ is Henkin domain invariant relative to $T$.
\end{enumerate}
\end{cor}
\begin{proof}
This follows from Definition~\ref{defn:bounded}, Theorem~\ref{thm:stable}, and Definition~\ref{defn:HDI}.
\end{proof}

If one restricts to standard models, one gets something more, namely that all bounded definable generalized quantifiers, regardless of the complexity of the definition, are domain invariant across standard models:
\begin{cor}\label{cor:stable:standard} (Bounded definability and domain invariance on standard models).

Suppose that $T$ is a theory which is true on all standard models. Suppose that $Q(\overline{R},\overline{R}^{\prime})$ is a generalized quantifier which is bounded definable relative to $T$. Suppose that $\mathcal{A},\mathcal{B}$ are two standard models with $\mathcal{A}$ a substructure of $\mathcal{B}$. Suppose that $\overline{R},\overline{R}^{\prime}$ are in $S_{\overline{n}}[A]$. Then one has $\mathcal{A}\models Q(\overline{R},\overline{R}^{\prime})$ iff $\mathcal{B}\models Q(\overline{R},\overline{R}^{\prime})$.
\end{cor}
\begin{proof}
This follows from Definition~\ref{defn:bounded} and  Corollary~\ref{cor:stable:standard:internal} and the fact that  on $\mathcal{B}$ we have that $\big(\Phi^{\mathscr{U}\overline{R}}(\overline{R},\overline{R}^{\prime})\big)^A$ is equivalent to $\Phi^{\mathscr{U}\overline{R}}(\overline{R},\overline{R}^{\prime})$ since $\overline{R}$ is in $S_{\overline{n}}[A]$.
\end{proof}

The previous corollary indicates that domain invariance across standard models is compatible with arbitrarily complex bounded definable formulas. As we will see in next section, things turn out quite differently with Henkin models.

We close by noting that branching quantifiers like Example~\ref{ex:branchmost} need not be $\Delta^1_1$-definable relative to $\mathsf{\Delta^1_1\mbox{-}CA}_0^-$:\footnote{Equation~(\ref{eqn:rebarwisebranching}) is from \cite[363, 494]{Peters2008aa}.}
\begin{prop}\label{prop:rebarwisebranching} (Branching quantifiers are not necessarily bounded $\Delta^1_1$-definable).

Suppose that $Q_1, \ldots, Q_k$ are $k$-many type $(1,1)$ generalized quantifiers. Consider the associated branching quantifier of type $(1, \ldots, k)$:
\begin{equation}\label{eqn:rebarwisebranching}
\mathsf{Br}_{\overline{Q}}(\overline{A},R) \mbox{ iff } \exists \; B_1\subseteq A_1, \ldots, B_k\subseteq A_k \; \bigwedge_{i=1}^k Q_i(A_i,B_i) \wedge B_1\times \cdots \times B_k\subseteq R
\end{equation}
If $Q_1, \ldots, Q_k$ are bounded $\Sigma^1_1$-definable relative to theory $T$, then so is $\mathsf{Br}_{\overline{Q}}$. But there are examples of $k\geq 1$ and $Q_1, \ldots, Q_k$ which are bounded $\Sigma^1_1$-definable relative to $\mathsf{\Delta^1_1\mbox{-}CA}_0^-$ such that $\mathsf{Br}_{\overline{Q}}$ is not bounded $\Delta^1_1$-definable relative to $\mathsf{\Delta^1_1\mbox{-}CA}_0^-$.
\end{prop}
\begin{proof}
Consider $k=1$ and $Q(X,Y)$ which says that $X\cap Y$ is Dedekind infinite (cf. Definition~\ref{defn:cardinality}). Then $Q(X,Y)$ is a bounded $\Sigma^1_1$-formula.

Let $\mathcal{A},\mathcal{B}$ be from Example~\ref{ex:disjunction:2}. That is, $\mathcal{A}$ is the non-standard Henkin model where $A$ is the complex numbers $\mathbb{C}$ with constant symbols for the field structure and with $S_n[A]$ being the first-order definable subsets of the $n$-th Cartesian product $\mathbb{C}^n$ of the complex field. Further, $\mathcal{B}$ is the standard model such that $B$ is the complex numbers $\mathbb{C}$ with constant symbols for the field structure.

Let $A,R$ be equal to $\mathbb{C}$. Then $\mathcal{B}\models \mathsf{Br}_Q(A,R)$ since in the metatheory one can identify the natural numbers as a subset of the complex field. However, $\mathcal{A}\models \neg \mathsf{Br}_Q(A,R)$ by the Ax-Grothendieck Theorem (cf. footnote~\ref{fnax}).
\end{proof}

\section{Statement of main results}\label{sec:statement}

The first main result says that Henkin domain independent generalized quantifiers are low level definable:

\begin{restatable}{thm}{thmfefermanesq}\label{thm:fefermanesq} (Grades of Henkin domain independence entail grades of predicative definability).

Suppose that $T$ is a theory extending  $\schoice$ and that $Q$ is a generalized quantifier. Then:

\begin{enumerate}[leftmargin=*]
    \item \label{thm:fefermanesq:1} If $Q$ is Henkin upward independent relative to $T$, then~$Q$ is $\Sigma^1_1$-definable relative to $T$.
    \item \label{thm:fefermanesq:2} If $Q$ is Henkin downward independent relative to $T$, then $Q$ is $\Pi^1_1$-definable relative to $T$.
    \item \label{thm:fefermanesq:3} If $Q$ is Henkin domain independent relative to $T$, then~$Q$ is $\Delta^1_1$-definable relative to $T$.
\end{enumerate}
\end{restatable}

We prove Theorem~\ref{thm:fefermanesq} in \S\ref{sec:proofmain}, and the proof follows Marker's proof of \cite{Marker1984aa} of Feferman's Preservation Theorem~\cite{Feferman1968-fi}. See the outset of \S\ref{sec:proofmain} for more discussion of Marker's proof.

Note that the converse of Theorem~\ref{thm:fefermanesq} is false, as one can see from Example~\ref{ex:disjunction:1}.

In his later work on the Tarski-Sher logicality thesis, Feferman himself mentions his earlier Preservation Theorem. First, in Feferman~\cite[\S{5}]{Feferman2010-gc}, he mentions it in passing in motivating a discussion of Manders' result that absoluteness conditions will force first-order definability in certain circumstances. Second, in Feferman~\cite[24]{Feferman2015-yr}, he appeals to his Preservation Theorem in order to prove that if a generalized quantifier can be implicitly defined as the unique solution to a sentence, then the generalized quantifier is first-order definable. This seems interesting and worthy of more discussion than I am able to give it here. But nothing in domain independence requires logicality, such as the type (1) quantifier in Example~\ref{ex:montagueindivdiaul} involving a named individual (but see footnote~\ref{fn:john}).

We can improve Theorem~\ref{thm:fefermanesq} by obtaining bounded predicative definability in certain situations. In order to describe these situations, we need one last definition.

\begin{defn}\label{defn:strongexistential} (Strong existential import).

Suppose that $T$ is a theory extending $\mathsf{ACA}_0^{-}$ and suppose that $Q(\overline{R},\overline{R}^{\prime})$ is a generalized quantifier. Then $Q$ has \emph{strong existential import relative to $T$} if $T$ proves 
\begin{equation}\label{eqn:strongexistential1}
\forall \; \overline{R} \; \forall \; \overline{R}^{\prime} \; \big(Q(\overline{R}, \overline{R}^{\prime})\rightarrow \mathscr{U}\overline{R}\neq \emptyset\big)
\end{equation}
and if for every individual constant $c$ in the signature, $T$ proves:
\begin{equation}\label{eqn:strongexistential2}
\forall \; \overline{R} \; \forall \; \overline{R}^{\prime} \; \big(Q(\overline{R}, \overline{R}^{\prime})\rightarrow (\mathscr{U}\overline{R})c \big)
\end{equation}
\end{defn}

To secure (\ref{eqn:strongexistential1}) one can of course define $Q^{\prime}(\overline{R}, \overline{R}^{\prime})$ to be $Q(\overline{R}, \overline{R}^{\prime})\wedge \mathscr{U}\overline{R}\neq \emptyset$. This is how the Aristotelian universal (Example~\ref{ex:allar}) is defined from the ordinary universal (Example~\ref{ex:all}). Likewise, note that (\ref{eqn:strongexistential2}) is satisfied by the paradigmatic example of a generalized quantifier defined with a constant symbol, namely the example of Montagovian individuals (Example~\ref{ex:montagueindivdiaul}). The general reason why we need to restrict to generalized quantifiers with strong existential import is that the empty set interacts poorly with binding (cf. Remark~\ref{rmk:binding:poorly}) and because substructures need to contain the denotations of the individual constant symbols (cf. Definition~\ref{defn:substructure}). 

Restricting to generalized quantifiers with strong existential import does not seem like a deep restriction given the way in which $Q(\emptyset, \overline{R}^{\prime})$ is controlled entirely by $Q(\emptyset, \emptyset)$ for Henkin conservative quantifiers (cf. Remark~\ref{rmk:consequence:henkinempty}). In most cases, one anticipates that either (i)~$Q(\emptyset, \emptyset)$ will be derivable from the background theory $T$ or (ii)~$\neg Q(\emptyset, \emptyset)$ will be derivable from the background theory $T$.  In case~(i), we simply define $Q^{\prime}(\overline{R}, \overline{R}^{\prime})$ by $Q(\overline{R}, \overline{R}^{\prime})\wedge \mathscr{U}\overline{R}\neq \emptyset$ and work with it instead. In case~(ii), we get an entailment to (\ref{eqn:strongexistential1}) when $Q$ is Henkin conservative with respect to the background theory $T$  (cf. Remark~\ref{rmk:consequence:henkinempty}). Note that in Examples~\ref{ex:some}-\ref{ex:usually} we are in case~(ii) for the minimal background theory $\mathsf{ACA}_0^{-}$.

Here is then the theorem on bounded definability:
\begin{restatable}{thm}{thmfefermanesqfour}\label{thm:fefermanesqfour} (Sufficient conditions for bounded $\Delta^1_1$-definability).

Suppose that $T$ is a theory extending $\schoice$ in a signature with only finitely many individual constants which is closed downwards under internal substructures. Suppose that $Q$ is a generalized quantifier which is Henkin domain independent relative to $T$ and Henkin conservative relative to $T$ and has strong existential import relative to $T$. Then~$Q$ is bounded $\Delta^1_1$-definable relative to $T$.
\end{restatable}

The proof of Theorem~\ref{thm:fefermanesqfour} is given at the end of \S\ref{sec:proofmain}. Note that a partial converse to this Theorem is in Corollary~\ref{cor:stable}. However, there are many bounded formulas which are not Henkin conservative relative to natural theories, and so a full converse which goes from bounded definability to Henkin conservativity is not possible.

It is no accident that domain independence and conservativity result in predicative definitions. For, historically a motivating idea of predicative constraints was that these definitions were stable across expansions,\footnote{For references to, and discussion of, Poincar\'e, Weyl, and Kreisel, see \cite[\S{2.3}]{Dean2017-mv} and \cite[\S{4}]{Walsh2016-uo}.} and Theorem~\ref{thm:fefermanesq}(\ref{thm:fefermanesq:3}) and Theorem \ref{thm:fefermanesqfour} give vivid displays of how certain kinds of stability entail varieties of $\Delta^1_1$-definability, which is the most traditional formal precisification of predicative constraints. By Theorem~\ref{thm:stable}, one can of course see that certain kinds of predicative definability entail a kind of stability or absoluteness. The surprise of course in Theorems~\ref{thm:fefermanesq}, \ref{thm:fefermanesqfour} is that this absoluteness requires definability of any kind.

I have not seen anything like Theorem~\ref{thm:fefermanesqfour} in the literature. However, in \cite{van-Benthem1984-wp}, van Benthem posed the problem of ``find[ing] a preservation result characterizing conservativity of first-order sentences'' (\cite[462]{van-Benthem1984-wp}, cf. \cite[38]{van-Benthem1986-qu}). He then says: ``The obvious conjecture is that $Q(X,Y)$ is conservative in $X$ iff $Q$ is logically equivalent to some sentence with all quantifiers $X$-restricted. Kit Fine observed that this follows, indeed, from the work of Feferman'' where Feferman's \cite{Feferman1968-fi} is cited (notation changed in quotation to match my own notation). But this discussion seems to assume at the outset that~$Q$ is first-order definable, whereas Theorem~\ref{thm:fefermanesqfour} is compatible with~$Q$ merely being $\Delta^1_1$-definable, as will be the case with the cardinality quantifiers (cf. Theorem~\ref{thmmostdefinabel}). I do not know anywhere that van Benthem or Fine say how they take Feferman \cite{Feferman1968-fi} to help establish the stated results for first-order definable $Q$. That said, the proof of Theorem~\ref{thm:fefermanesqfour}, given at the close of \S\ref{sec:proofmain}, is comparatively direct once the definitions are in place.

\section{Applications to explanatory projects in formal semantics}\label{sec:discuss}

The semantics of generalized quantifiers are again given by a collection of tuples of subsets of Cartesian products of the underlying domain.\footnote{This is a key tenant of the tradition of generalized quantifiers (cf. footnote~\ref{fn:one}), but also present in introductory semantics textbooks, cf. \cite[140 ff]{Heim1998-bs}, \cite[\S{7.2}, pp. 223 ff]{Gamut1991-gc}.} A chief challenge is to explain why so few of these are realized in natural language. For instance, Sternefield writes in a recent survey: ``the number of possible quantifiers left as a result of such axioms as (17) [conservativity] is still huge. [\ldots] there is nothing in the theory that would allow for isolating exactly the `primitive' or even the `lexical' quantifiers of a natural language'' (\cite[11]{Sternefeld2020-vf}). Similarly, von Fintel-Matthewson say that  ``It is clear that natural languages do not nearly make full use of the logically possible space of determiner meanings. The conservativity universal may be fraying at the edges, but the investigation of the limits on possible determiner
meanings is still a promising avenue of research'' (\cite[163-164]{von-Fintel2008-tz}).

The main reason why Theorems~\ref{thm:fefermanesq}, \ref{thm:fefermanesqfour} are significant is that they explain why so few options for generalized quantifiers are realized in natural language. Namely, if one extends the domain independence conditions and the conservativity conditions to Henkin models, then these theorems directly entail the comparative low-level definability of the generalized quantifiers. Or, if one thought that domain independence and conservativity were just as much in need of explanation as the limited options for generalized quantifier meanings, then these theorems say that any explanation of domain independence and conservativity that also works for Henkin models itself suffices to explain the limited options for generalized quantifier meanings; and hence these explanatory tasks are not independent of one another.

Before turning to whether the motivations for the semantic constraints extend to Henkin models, it is worth being a little more explicit about what it means to say that so few of the theoretically possible options for generalized quantifiers are realized in natural language. This should not be understood as a claim about what generalized quantifiers appear in the lexicon of natural languages. For, the lexicon will always be limited and will face pressure to be sparse, and what a sparse lexicon can generate can be vast: for instance, generalized quantifiers are closed under Boolean operations.\footnote{\cite[91-92]{Peters2008aa}, \cite[\S{4.1.1}]{Szabolcsi2010-vi}. This has a prominent place in \cite[265 ff]{Keenan1986-pq}.} I take it rather that the following is intended: there are some determiners which we recognize as generalized quantifiers, and of this subclass of determiners, their truth-conditions seem to be a comparatively small slice of those which are theoretically possible, given that generalized quantifiers have as their meanings simply arbitrary collections of tuples of subsets of Cartesian products of the underlying domain.

It is however important to be precise about the meaning of ``comparatively small.''  On the one hand, it might be understood in terms of \emph{numerical constraints}. For instance, if a standard model has first-order domain with $n$ elements, then there are $2^{2^n\times 2^n}$ many theoretically possible denotations for type $(1,1)$ quantifiers, but only $2^{3^n}$ many theoretically possible denotations for type $(1,1)$ quantifiers which are conservative.\footnote{\cite[327]{Peters2008aa}. For considerations like this, see \cite[290]{Keenan1986-pq}, \cite[9]{van-Benthem1986-qu}, \cite[\S{4}, p. 54]{Keenan1996-yi}, \cite[10]{Sternefeld2020-vf}.}  Alternatively, it might be understood in terms of \emph{definability constraints} on the truth conditions for these quantifiers. That is, in inspecting the list of examples of generalized quantifiers in \S\ref{sec:genq}, one notes that they all seem to be easy or simple to define, and definitions featuring e.g. 27 second-order quantifiers appear to be absent. Obviously Theorems~\ref{thm:fefermanesq}, \ref{thm:fefermanesqfour} fare better in contributing to the explanatory task at hand when one works with definability constraints rather than numerical constraints.

But one's applications of Theorems~\ref{thm:fefermanesq}, \ref{thm:fefermanesqfour} will vary depending on what one thinks of Example~\ref{ex:branchmost}, our paradigmatic example of a branching quantifier (cf. (\ref{eqn:rebarwisebranching})). Examples like these are often bounded definable, and by  boundedness, they are domain invariant across standard models by Corollary~\ref{cor:stable:standard}, and they are often Henkin conservative and have strong existential import relative to minimal and natural theories. But Theorems~\ref{thm:fefermanesq},\ref{thm:fefermanesqfour} indicate that they are Henkin domain invariant relative to a theory only if they are bounded $\Delta^1_1$-definable relative to the theory. Some of them are demonstrably not so definable (cf. Proposition~\ref{prop:rebarwisebranching}), and for others one may suspect that are not so definable. For someone who took Example~\ref{ex:branchmost} as a core rather than periphery case of a generalized quantifier, this would be evidence that not all generalized quantifiers are downwards Henkin invariant. This is compatible with still holding that generalized quantifiers of natural languages are Henkin upwards invariant as well as Henkin conservative. From the perspective of someone who held this, Theorems~\ref{thm:fefermanesq}(\ref{thm:fefermanesq:1}) could still well serve to explain why so few theoretically possible generalized quantifiers appear in natural language, but with the definability constraint set at $\Sigma^1_1$-definable rather than bounded $\Delta^1_1$-definable.

This discussion also gestures towards a possible method, which is simultaneously empirical and mathematical, of refuting the contention that the generalized quantifiers of natural language are Henkin domain invariant as well as Henkin conservative. Namely, one could try to produce an example of something which was recognizably a generalized quantifier of natural language and yet which was demonstrably not bounded $\Delta^1_1$-definable relative to some natural theory. Likewise, apropos the previous paragraph, one could attempt to produce an example, like Example~\ref{ex:branchmost}, at the periphery of the generalized quantifiers of natural language, but which was yet demonstrably not $\Sigma^1_1$-definable relative to some natural theory. Such examples would of course be of further interest since they do not occur in the lists of the generalized quantifiers of natural language in the standard surveys.\footnote{Cf. footnote~\ref{fn:one}}$^{\mbox{,\;}}$\footnote{For a possible source of such examples, see the Keenan-Stavi Theorem discussed at the close of \S\ref{sec:further}.}

\section{Motivations for semantic constraints extend to Henkin models}\label{sec:discuss:2}

Given the contributions to the explanatory projects discussed in the previous section, a real question then is whether the motivations for insisting on domain independence and conservativity extend from standard models to Henkin models. This seems very clear in the case of conservativity. For, conservativity is just a biconditional, namely in the case of type $(1,1)$ quantifiers, the biconditional $Q(X,Y)\leftrightarrow Q(X,X\cap Y)$.\footnote{In the general case, it is the biconditional~(\ref{eqn:henkinconservation}) from Definition~\ref{defn:Hcon}.} Since this is expressible in the object language of second-order logic and hence does not distinguish between Henkin models and standard models, it seems that whatever motivates it or explains it will not distinguish between these two classes of models. For instance, conservativity is often explained by natural language needing a mechanism by which to temporarily introduce a domain of discourse. Keenan calls conservativity the ``domain setting'' property of a generalized quantifier.\footnote{\cite[1066]{Keenan2011-pt}. Likewise, von Fintel-Matthewson write: ``What this [conservativity] means is that the first argument of a determiner `sets the scene'~'' (\cite[161]{von-Fintel2008-tz}).} If this domain setting role is accomplished by the aforementioned biconditional, then it can be done in both Henkin and standard models. However, conservativity is sometimes explained by appealing to structural features of syntactic processing.\footnote{This is the kind of explanation offered by Romoli \cite{Romoli2010-ra}, \cite{Romoli2015-qo}. Romoli is building off of prior work by Chierchia, Fox, Ludlow, and Sportiche. See \cite[\S{4.2.2.2}]{Szabolcsi2010-vi} for a survey of some of this work.} In this case, to the extent that the syntactic processing is constrained to act primarily on the logical forms correlated with $Q(X,Y)$ and $Q(X,X\cap Y)$, it seems this will not distinguish between Henkin and standard models.

As for domain independence, Peters and Westerst{\aa}hl describe it as follows:
\begin{quote}
Roughly, it is the property that the behavior of the quantifier doesn't change when you extend the universe (\cite{Peters2008aa} p. 100).
\end{quote}
The lack of change here of course has to be understood to mean that subsets which are in the denotation of the quantifier prior to the extension (or contraction) of the universe continue to be in the denotation of the quantifier after the extension (or contraction). Hence, the question is whether any of the reasons we have for insisting on this lack of change in standard models carries over to Henkin models. All finite models which satisfy first-order comprehension are standard models, and so looking only at finite models does not distinguish between the Henkin and standard semantics. But if one wants a semantics with a Soundness and Completeness Theorem and with the recursive enumerability of the validity relation, then the Henkin semantics are the clear choice.\footnote{By Trakhtenbrot's Theorem (e.g. \cite[Theorem X.5.4]{Ebbinghaus2021-jl}) validity on finite models is not recursively enumerable and encodes the complement of the halting set. Hence to maintain recursive enumerability of the validity relation, one has to countenance at least some infinite models. By using Dedekind finiteness, as we do in \S\ref{sec:cardinality}, one can still work with the Henkin semantics and a certain kind of finiteness postulate.} Further, experience suggests that many arguments involving standard models can be recast with minimal modification as derivations in second-order logic which are hence applicable to Henkin models. For, one simply collects together the hypotheses about models to which one is appealing, and if they simply speak about elements of the underlying domain and relations on the underlying domain, then they can be regimented into sentences in the object-language of second-order logic.\footnote{For instance, \cite[Part II]{Button2018-ux} explores how many authors have recast the famous categoricity theorems of second-order arithmetic and set theory as derivations in second-order logic.}

Peters and Westerst{\aa}hl also describe domain independence as follows:
\begin{quote}
What, then, does \emph{constancy} of a quantifier $Q$ mean? Intuitively, it means that $Q$ `means the same', or is \emph{given by the same rule}, on every universe. But the identity conditions for meanings, or for rules, are notoriously difficult to lay down. Indeed, it is not clear that the concept of constancy can be explained fully in an extensional framework like ours. But some things can be said. [\P] A special case is when the rule does not mention the universe $M$ at all. This seems
to be what $\mathsf{Ext}$ amounts to (\cite{Peters2008aa} p. 106).
\end{quote}
If `given by the same rule' simply means `given by a single rule definable in second-order logic,' as in Definition~\ref{defn:gq}, then this is very undemanding requirement, and will not get one to domain independence, in either the standard or Henkin semantics. But the latter part of the quotation seems to get closer to the core idea: namely a quantifier is domain independent if it does not query the underlying domain in the wrong way.

This can be borne out by reflecting back on the type $(1,1)$ generalized quantifier $Q(X,Y)$ from Example~\ref{ex:disjunction:1}, which is the existential quantifier if there are fewer than three objects in the underlying domain but is the universal quantifier if there are at least three objects in the underlying domain. This disjunctive example acts differently on $X,Y$ depending on how things stand in other first-order parts of the universe; and as we saw in Example~\ref{ex:disjunction:1}, it is not domain independent. This motivation then extends clearly to Henkin models, and just as one wants to restrict attention to type $(1,1)$ quantifier $Q(X,Y)$ which leave the focus on $X,Y$ and do not query other first-order parts of the universe, so one would want them to not query other second-order parts of the universe. This was illustrated by Examples~\ref{ex:disjunction:2}, \ref{ex:disjunction:3}, where common idea was that the quantifier was the existential quantifier if the second-order part of the universe was expressively impoverished and it was the universal quantifier if the second-order part of the universe was expressively rich.

Hence the suggestion is that one should accept both domain independence and Henkin independence of generalized quantifiers because one wants to restrict attention to type $(1,1)$ generalized quantifiers $Q(X,Y)$ whose satisfaction conditions pertain to $X,Y$ and not to other regions of the first- or second-order domain. (For notational ease, we continue to focus the discussion on type $(1,1)$ generalized quantifiers; but this applies \emph{mutatis mutandis} for generalized quantifiers of other types). An advantage of this proposal is that brings the motivation for domain independence closer to the motivation for conservativity.\footnote{There is precedent for bringing them closer in Keenan \cite{Keenan1996-yi}, Szabolcsi \cite[\S{4.2.2.3}]{Szabolcsi2010-vi}, Glanzberg \cite[803]{Glanzberg2006-is}. For instance, Keenan writes that ``It [the role of the noun argument] serves to limit the domain of objects we use the predicate to say something about. This simple idea of domain restriction is captured in the literature with two independent constraints: \emph{conservativity} and \emph{[domain]-extension}'' (\cite[\S{4}, p. 54]{Keenan1996-yi}). Later in the same section, he says more about domain limitation, to wit: ``But natural languages are general purpose. Speakers use them to talk about anything they want, and common nouns in English provide the means to delimit `on line' what speakers talk about and quantify over'' (\cite[\S{4}, p. 56]{Keenan1996-yi}).} Recall that the latter was a ``domain setting'' feature, where the idea is that $Q(X,Y)$ results in making $X$ the temporary domain of discourse to which $Y$ is restricted. Our discussion of the motivation of domain independence suggests the further thought: in addition to considering only $Y$'s intersection with $X$, one must \emph{not} query either first-order or second-order facts about the model which are outside of $X$.

However, the proposal does come with disadvantages and challenges. First, the proposal introduces more explicitly notions of ``aboutness'' or ``topicality'' into this understanding of the motivations for domain independence and conservativity.\footnote{Szabolcsi is explicit about this, writing: ``The NP-set serves as the determiner's topic in the intuitive sense'' (\cite[\S{4.2.2.3}]{Szabolcsi2010-vi}).}$^{\mbox{,\;}}$\footnote{One might object that nothing like this needed, and one can just get by with binding to and intersecting with $X$ in a type $(1,1)$ quantifier $Q(X,Y)$. While this freezes the first-order objects as one moves between Henkin models, it will not freeze the second-order entities. That said, there are more radical ways to freeze the second-order entities: namely, one could just restrict throughout to standard models. But that would come with the cost of making the validity relation massively intractable. However, there may be more nuanced proposals in the offing for how, in theory and practice, we freeze the second-order domain.} It is well-known that this kind of intentionality is hard to explicate and model, at least when one restricts oneself to the tools of formal semantics and logic. Second, the proposal conceives of a core task of domain independence as getting rid of the disjunctive examples in \S\ref{sec:disjunctiveexamples}. But the history of Nelson's grue should make us cautious in claiming anything more than provisional success in weeding out disjunctive laws. Third, a presupposition of the above discussion is that linguistic agents are moving between different Henkin models in the same way that they are moving between contexts in a more quotidian sense, and that the semantic constraints on quantifiers allows agents to hone in and focus on a topic that is fixed amid this wider flux. It does seem that some variation in the second-order domain is quotidian: in some discussions a limited conception of properties is at issue and in other discussions a more fine-grained conception of properties is at issue. That said, one would want a better understanding of the circumstances in which the second-order domain becomes subject to substantial variation, and whether the variation could be understood like Examples~\ref{ex:disjunction:2}-\ref{ex:disjunction:3} in which one goes back and forth between a mathematically impoverished and a mathematically enriched setting.

\section{Lowering to first-order definability}\label{sec:lower}

The theorems discussed thus far, namely Theorems~\ref{thm:fefermanesq}, \ref{thm:fefermanesqfour}, get one to bounded $\Delta^1_1$-definability. But the standard list of examples, such as we sampled in \S\ref{sec:genq}, consists in large part of generalized quantifiers which are first-order definable. In this section we show that by switching to a slightly richer background theory we can improve the definability to first-order definability.

We expand from pure second-order logic to a well-order:\footnote{This does two things for us. First, it gives us definable Skolem functions (cf. Definition~\ref{defn:skolem} and Example~\ref{ex:skolem}), which we need for the proof of Theorems~\ref{thm:fodefinable}, \ref{thm:fodefinabletwo}, given in \S\ref{sec:fodefinable}; and it is a natural case of definable Skolem functions which results in a theory which is downward closed under internal substructures. Second, it allows us to prove results about cardinality in the Dedekind finite setting, which we discuss in \S\ref{sec:cardinality}, and prove in \S\ref{sec:cardinalityqs}.}
\begin{defn}\label{defn:two} (Well-order notions and axiomatizing finitude).

Let $T^{-}$ be one of $\aca$, $\dcomp$, $\schoice$, $\scomp$, or $\sol$.

Then we define $T^{\mathsf{wo}}$ to be the expansion of $T^{-}$ by the axioms saying that a fresh binary relation symbol $\preceq$ is a well-order of the first-order domain:
\begin{itemize}[leftmargin=*]
    \item \emph{Reflexivity}: $\forall \; x \; x\preceq x$
    \item \emph{Anti-symmetry}: $\forall \; x \; \forall \; y \; ((x\preceq y \; \& \; y\preceq x)\rightarrow x=y)$
    \item \emph{Transitivity}: $\forall \; x \; \forall \; y \; \forall \; z \; ((x\preceq y \; \& \; y\preceq z) \rightarrow x\preceq z)$
    \item \emph{Linearity}: $\forall \; x \; \forall \; y \; (x\preceq y \vee y\preceq x)$
    \item \emph{Well-ordering}: $\forall \; X \; ((\exists \; x \; Xx)\rightarrow (\exists \; x_0 \; Xx_0 \; \& \; \forall \; y \; (Xy\rightarrow x_0\preceq y)))$
\end{itemize}
Finally $T^{\mathsf{fin}}$ is the expansion of $T^{\mathsf{wo}}$ by the axiom which says that universe $\mathbb{V}$ is Dedekind finite (cf. Definition~\ref{defn:cardinality}).
\end{defn}

Note that if  $T^{-}$ be one of $\aca$, $\dcomp$, $\schoice$, $\scomp$, or $\sol$, then both $T^{\mathsf{wo}}$ and $T^{\mathsf{fin}}$ are closed downwards under internal substructures (cf. Definition~\ref{defn:closeddownardinternal}), since initial segments of well-orders are still well-orders, and since Dedekind infinitude is bounded $\Sigma^1_1$-definable.

Then we have the following results: 

\begin{restatable}{thm}{thmfodefinable}\label{thm:fodefinable}
(From $\Delta^1_1$-definability to first-order definability).

Suppose that $Q$ is a generalized quantifier which is $\Delta^1_1$-definable relative to $\schoicewo$. Then $Q$ is first-order definable relative to $\schoicewo$.

\end{restatable}

\begin{restatable}{thm}{thmfodefinabletwo}\label{thm:fodefinabletwo}
(Sufficient condition for bounded first-order definable).

Suppose that $Q$ is a generalized quantifier which is Henkin domain invariant relative to $\schoicewo$ and Henkin conservative relative to $\schoicewo$ and has strong existential import relative to relative to $\schoicewo$. Then $Q$ is bounded first-order definable relative to $\schoicewo$.
\end{restatable}

We prove these two theorems in \S\ref{sec:fodefinable}. However, note that Theorems~\ref{thm:fodefinable}-\ref{thm:fodefinabletwo} pertain to a specific theory, namely $\schoicewo$, whereas Theorems~\ref{thm:fefermanesq}, \ref{thm:fefermanesqfour} are about general background theories $T$ of second-order logic.

\section{Quantifiers defined in terms of cardinality}\label{sec:cardinality}

Finally, we make good on an obligation to show that Examples~\ref{ex:most}, \ref{ex:more} of $\mathsf{Most}$ and $\mathsf{More\mbox{-}than}$ fit into the framework described in this paper. We have the following, where $\scompfin$ was defined at the end of Definition~\ref{defn:two}:

\begin{restatable}{thm}{thmmostdefinabel}\label{thmmostdefinabel} (A theory for the paradigmatic cardinality-defined quantifiers).

\begin{enumerate}[leftmargin=*]
    \item\label{thmmostdefinabel:1} $\mathsf{Most}$ and $\mathsf{More\mbox{-}than}$ are bounded $\Delta^1_1$-definable relative to $\scompfin$.
    \item\label{thmmostdefinabel:2} $\mathsf{Most}$ and $\mathsf{More\mbox{-}than}$ are Henkin domain independent and Henkin conservative relative to $\scompfin$.
\end{enumerate}

\end{restatable}

We prove this in \S\ref{sec:cardinalityqs}. Obviously (\ref{thmmostdefinabel:1}) follows from (\ref{thmmostdefinabel:2}) and Theorems~\ref{thm:fefermanesq}, \ref{thm:fefermanesqfour}. But to our knowledge, there is no easy way to verify the Henkin domain independence in (\ref{thmmostdefinabel:2}) without going through the specific definitions produced in the proof of (\ref{thmmostdefinabel:1}) and then appealing to Theorem~\ref{thm:stable}.

It is natural to ask whether the restriction to Dedekind finite universes in Theorem~\ref{thmmostdefinabel} is necessary. Let us consider the case of $\mathsf{More\mbox{-}than}$ (cf. Example~\ref{ex:more}). Note that one has $\left|X\right|>\left|Y\right|$ iff $\mathsf{More\mbox{-}than}(X,Y,\mathbb{V})$. Hence if  $\mathsf{More\mbox{-}than}$ is bounded $\Delta^1_1$-definable relative to a theory $T$, then $\left|X\right|>\left|Y\right|$ is $\Delta^1_1$-definable relative to $T$, by quantifiers which are bound to $X\cup Y$ (in the sense of Definition~\ref{defn:binding}(\ref{defn:binding2})).

But, one can show using the method of forcing that this is false for $\mathsf{ACA}_0^-$.\footnote{In the setting of set theory, one can study models of $\mathsf{ACA}_0^-$ which come equipped with a membership relation on individuals which satisfies the appropriate second-order axioms of set theory. This axiom system is usually known as von Neumann-G\"odel-Bernays set theory or $\mathsf{NBG}$ for short (cf. \cite[Chapter 4]{Mendelson1997-pp}).} For, using this method, one can produce two Henkin models $\mathcal{A}, \mathcal{B}$ of $\mathsf{ACA}_0^-$ such that $\mathcal{A}$ is a substructure of $\mathcal{B}$ and one can produce two $X,Y$ in $S_1[A]$ such that $\mathcal{A}\models \left|X\right|< \left|Y\right|$ but $\mathcal{B}\models \neg ( \left|X\right|< \left|Y\right|)$. Namely, $\mathcal{A}$ is the countable transitive ground model together with its first-order definable subsets, and $\mathcal{B}$ is a generic extension together with its first-order definable subsets; one further chooses $X=\omega$ and $Y=\kappa$ for some uncountable cardinal $\kappa$ in $\mathcal{A}$ which is collapsed in $\mathcal{B}$. If $\left|X\right|<\left|Y\right|$ was $\Delta^1_1$-definable relative to $\mathsf{ACA}_0^-$ with quantifiers bound to $X\cup Y$, then we would contradict Theorem~\ref{thm:stable}.

The restriction to the finite in Theorem~\ref{thmmostdefinabel} does not seem problematic, as far as empirical coverage goes. For, our judgments about $\mathsf{Most}$ and $\mathsf{More\mbox{-}than}$ are primarily informed by our experience with finite collections of e.g. people and other middle-sized objects. But the restriction to the finite does suggest an interesting test case for someone who wanted to deny that generalized quantifiers satisfy Henkin domain independence. For, by the above, $\mathsf{More\mbox{-}than}$ is not Henkin domain independent relative to $\mathsf{ACA}_0^-$. But it is domain independent on standard models by Corollary~\ref{cor:stable:standard}. If one thought that $\mathsf{More\mbox{-}than}$ was a generalized quantifier regardless of the size of the ambient universe, then this would be a reason to reject that all generalized quantifiers satisfy Henkin domain independence. Pursued in this direction, the discussion would start blending with discussions in philosophy of mathematics about whether and to what extent models produced by the method of forcing are or are not good candidates for intended models of set theory.\footnote{Categoricity arguments have been used, not uncontroversially, to answer ``no they are not''; see \cite[Chapter 11]{Button2018-ux} for references to, and discussion of, Kreisel, McGee, and Martin among others. For nuanced versions of the answer ``yes they are,'' see Hamkins \cite{Hamkins2012-ld}, although some care would have to exercised to situate his discussion of the models produced by forcing into the discussion of second-order logic.} But those who outright reject the Henkin domain independence of generalized quantifiers would be in need of another explanation of why natural language generalized quantifiers realize so few of the theoretically possible options for collections of subsets of the underlying domain.

\section{Applications to upper bounds on complexity}\label{sec:applications}

Finally, we mention an immediate corollary of Theorem~\ref{thm:fefermanesq}(\ref{thm:fefermanesq:3}) (there are similar corollaries for (\ref{thm:fefermanesq:1})-(\ref{thm:fefermanesq:2})). (Recall the axioms for well-orders in Definition~\ref{defn:two}).
\begin{cor}\label{cor:logicapps} (Upper bounds on complexity).
\begin{enumerate}[leftmargin=*]
    \item\label{cor:logicapps:2} Suppose that $T$ is a theory extending $\schoice$ containing an axiom saying that a binary constant symbol is a well-order of the domain. Suppose $T$ is true on all standard models with first-order domain which is countably infinite and with the constant symbol interpreted as a well-order. Suppose that $Q$ is a generalized quantifier which is Henkin domain independent with respect to $T$. For each countably infinite set $A$, consider the associated standard model~$\mathcal{A}$ of $T$ along with a well-order. For each $n\geq 1$, the set $P(A^n)$ is naturally a Polish space by identifying it with the countable product $\{0,1\}^{A^n}$, and similarly for finite products of these. Then the quantifier $Q$ is Borel on these models, under this identification.
    \item\label{cor:logicapps:1} Suppose that $T$ is a theory extending  $\schoice$ containing an axiom saying that a binary constant symbol is a well-order of the domain. Suppose that $T$ is true on all standard models with finite first-order domain $\{0, \ldots, m\}$ for $m\geq 1$ and with the constant symbol interpreted as the usual ordering relation on the natural numbers. Suppose that $Q$ is a generalized quantifier which is Henkin domain independent with respect to $T$. For $m\geq 1$, consider the finite standard models $\mathcal{A}_m$ of $T$ with domain $\{0, \ldots, m\}$ and with the constant symbol interpreted as the usual ordering relation for the natural numbers. Then the problem of determining membership in $Q$ on these models is $\mathsf{NP}\cap \mathsf{co\mbox{-}NP}$.
\end{enumerate}
\end{cor}
\begin{proof}
For (\ref{cor:logicapps:2}), by Theorem~\ref{thm:fefermanesq}(\ref{thm:fefermanesq:3}), this follows from Souslin's Theorem, which says that testing whether $\overline{R}$ is in a $\Delta^1_1$-definable set over a Polish space is Borel.\footnote{\cite[88]{Kechris1995-hr}.}

For (\ref{cor:logicapps:1}), by Theorem~\ref{thm:fefermanesq}(\ref{thm:fefermanesq:3}), this follows from Fagin's Theorem, which says that testing whether $\overline{R}$ satisfies a $\Delta^1_1$-definable problem over $\mathcal{A}_m$ is  $\mathsf{NP}\cap \mathsf{co\mbox{-}NP}$.\footnote{\cite[115]{Immerman1996aa}.}.
\end{proof}

This Corollary leverages how predicatively definable sets occur naturally throughout mathematics and computation and are tied to core notions like the Borel sets and the traditional complexity classes. Let us discuss each of these in turn.

Borel sets are the events most commonly used in probability, measure-theoretic mathematics, and their many applications. The subject-matter of descriptive set theory is definability in Polish spaces, where the Borel sets form the lower end of a hierarchy of definability.\footnote{Cf. \cite{Kechris1995-hr}.} Corollary~\ref{cor:logicapps}(\ref{cor:logicapps:2}) says that the quantifiers at issue satisfy the usual condition for the theorems of probability theory and measure-theoretic mathematics to apply to them, and that they are not a source of set-theoretic pathology. However, the applicability of this is limited. As we saw in \S\ref{sec:cardinality}, the theory $\mathsf{ACA}_0^-$, which is true on all standard models with first-order domain which is countably infinite, is such that the quantifier $\mathsf{More\mbox{-}than}$ is not Henkin domain independent relative to $\mathsf{ACA}_0^{-}$.

Computational complexity theory studies feasible models of computation, which admit computations that are a strict subclass of those allowed under the more traditional Turing models of computation.\footnote{Cf. \cite{Immerman1996aa}.} Corollary~\ref{cor:logicapps}(\ref{cor:logicapps:1}) identifies an upper bound on the computational complexity of the quantifier. If the motivations for domain independence extend to Henkin domain independence, then Corollary~\ref{cor:logicapps}(\ref{cor:logicapps:1}) would predict that natural language generalized quantifiers are not excessively infeasible to compute. Finally, in contrast to the conclusion of the previous paragraph, in \S\ref{sec:cardinality} we identified the theory $\scompfin$ which satisfies the hypotheses of Corollary~\ref{cor:logicapps}(\ref{cor:logicapps:1}) and relative to which $\mathsf{Most}$ and $\mathsf{More\mbox{-}than}$ are Henkin domain invariant and Henkin conservative. Hence $\mathsf{Most}$ and $\mathsf{More\mbox{-}than}$ are $\mathsf{NP}\cap \mathsf{co\mbox{-}NP}$ on the finite models described in Corollary~\ref{cor:logicapps}(\ref{cor:logicapps:1}).

Corollary~\ref{cor:logicapps}(\ref{cor:logicapps:1}) is rather complementary to Szymanik's work, which puts the upper bound at $\mathsf{NP}$ rather than $\mathsf{NP}\cap \mathsf{co\mbox{-}NP}$.\footnote{\cite[14]{Szymanik2016-ue}.} 
This is not lowered due to branching quantifiers and Ramsey quantifiers (cf. Example~\ref{ex:branchmost}).\footnote{\cite[\S{7.3},\S{7.5}]{Szymanik2016-ue}.} For someone who thought that branching quantifiers and Ramsey quantifiers were core rather than periphery generalized quantifiers, one would rather focus on the analogue of Corollary~\ref{cor:logicapps}(\ref{cor:logicapps:1})  which is centered around Fagin's Theorem and Theorem~\ref{thm:fefermanesq}(\ref{thm:fefermanesq:1}) rather than Theorem~\ref{thm:fefermanesq}(\ref{thm:fefermanesq:3}).

We now turn to the more technical part of the paper, in which we prove the various theorems in turn.

\section{Proofs of Theorem~\ref{thm:stable}, Theorem~\ref{thm:internal}}\label{sec:stable}

In this theorem, recall that the binding concepts and notation was defined in Definition~\ref{defn:binding}.

\thmstable*

\begin{proof}

By Proposition~\ref{prop:absolutenboolean} it suffices to consider the case where $X$ is a unary second-order variable and where $\varphi(\overline{x},\overline{R})$ is in the signature with just membership.

We begin by showing the biconditional for first-order $\varphi(\overline{x},\overline{R})$, by induction on complexity of $\varphi(\overline{x},\overline{R})$.

As a base case, suppose that that $\varphi$ is the atomic $R\overline{y}$, where $\overline{y}$ is some subtuple of $\overline{x}$. Since this atomic $\varphi$ does not have any quantifiers, one has that $\varphi^X$ is the same as $\varphi$. Since $\mathcal{A}, \mathcal{B}$ are Henkin models, one has that both $\mathcal{A}, \mathcal{B}$ interpret $R\overline{y}$ as the membership of $\overline{y}$ in $R$, and so we are done. The same is true if some of $R$ and the entries in $\overline{y}$ are constants, since this was built into the definition of substructure (cf. Definition~\ref{defn:substructure}).

The inductive steps for the propositional connectives are clear. 

We do the inductive step for the first-order existential quantifier. Suppose it holds for $\varphi(\overline{x},y,\overline{R})$; we show it holds for $\exists \; y \; \varphi(\overline{x},y,\overline{R})$. 
\begin{itemize}[leftmargin=*]
\item If $\mathcal{A}\models (\exists \; y \; \varphi(\overline{x},y,\overline{R}))^X$, then there is $y$ in $X$ such that $\mathcal{A}\models \varphi(\overline{x},y,\overline{R})^X$. Then $y$ is in $X\subseteq A\subseteq B$ and hence $\mathcal{B}\models Xy$  and by induction hypothesis $\mathcal{B}\models \varphi^X(\overline{x},y,\overline{R})$ and so $\mathcal{B}\models \exists \; y \; Xy \wedge \varphi^X(\overline{x},y,\overline{R})$ and so $\mathcal{B}\models (\exists \; y \; \varphi(\overline{x},y,\overline{R}))^X$.
\item If $\mathcal{B}\models (\exists \; y \; \varphi(\overline{x},y,\overline{R}))^X$, then $\mathcal{B}\models \exists \; y \; Xy \wedge \varphi^X(\overline{x},y,\overline{R})$. Then there is $y$ in $X$, which is a subset of $A$, such that $\mathcal{B}\models \varphi^X(\overline{x},y,\overline{R})$. By induction hypothesis, $\mathcal{A}\models \varphi^X(\overline{x},y,\overline{R})$. Then $\mathcal{A}\models (\exists \; y \; \varphi(\overline{x},y,\overline{R}))^X$.
\end{itemize}

To show the conditional~(\ref{eqn:stable:conditional}) for a $\Sigma^1_1$-formula, suppose for simplicity that it is the case of $\exists \; S \; \varphi(\overline{x},y,\overline{R},S)$, where $S$ is unary and $\varphi(\overline{x},y,\overline{R},S)$ is first-order. Suppose that $\mathcal{A}\models (\exists \; S \; \varphi(\overline{x},y,\overline{R},S))^X$. Then $\mathcal{A}\models \exists \; S \; S\subseteq X \; \varphi^X(\overline{x},y,\overline{R},S)$. Then by Proposition~\ref{prop:absolutenboolean}, there is $S$ in $S_1[A]\subseteq S_1[B]$ such that $S\subseteq X$ and $\mathcal{A}\models \varphi^X(\overline{x},y,\overline{R},S)$. By the biconditional for $\varphi(\overline{x},y,\overline{R},S)$, one has $\mathcal{B}\models  \varphi^X(\overline{x},y,\overline{R},S)$. Then by Proposition~\ref{prop:absolutenboolean} again, one has $\mathcal{B}\models \exists \; S \; S\subseteq X \; \varphi^X(\overline{x},y,\overline{R},S)$ and hence $\mathcal{B}\models (\exists \; S \; \varphi(\overline{x},y,\overline{R},S))^X$.

The reverse conditional~(\ref{eqn:stable:conditional}) for a $\Pi^1_1$-formula follows from the $\Sigma^1_1$ case by contraposition and DeMorgan.

\end{proof}

\thminternal*

\begin{proof}
Since $\mathcal{A}\models \varphi(\overline{x},\overline{R})$ iff $\mathcal{A}\models \varphi^A(\overline{x},\overline{R})$ the conclusion for first-order $\varphi(\overline{x},\overline{R})$ follows from the previous Theorem. Suppose it holds for $\varphi(\overline{x},\overline{R},S)$; we show it holds for $\exists \; S \; \varphi(\overline{x},\overline{R},S)$, where for the sake of simplicity $S$ is unary.
\begin{itemize}[leftmargin=*]
    \item Suppose that $\mathcal{A}\models \exists \; S \; \varphi(\overline{x},y,\overline{R},S)$. Then there is $S$ in $S_1[A]\subseteq S_1[B]$ such that $\mathcal{A}\models \varphi(\overline{x},y,\overline{R},S)$. Then $\mathcal{A}\models S\subseteq A$ and by two applications of Proposition~\ref{prop:absolutenboolean} we have $\mathcal{B}\models S\subseteq A$. Further, by induction hypothesis $\mathcal{B}\models \varphi^A(\overline{x},y,\overline{R},S)$. Hence $\mathcal{B}\models (\exists \; S \; \varphi(\overline{x},\overline{R},S))^A$.
    \item Suppose that $\mathcal{B}\models (\exists \; S \; \varphi(\overline{x},\overline{R},S))^A$. Then $\mathcal{B}\models \exists \; S\subseteq A \wedge \varphi^A(\overline{x},\overline{R},S)$. By Proposition~\ref{prop:absolutenboolean}, there is $S$ in $S_1[B]$ such that $S\subseteq A$ and $\mathcal{B}\models \varphi^A(\overline{x},\overline{R},S)$. Since $\mathcal{A}$ is an internal substructure of $\mathcal{B}$, one has that $S_1[A]=S_1[B]\cap P(A)$, and hence $S$ is in $S_1[A]$. And by induction hypothesis we have $\mathcal{A}\models \varphi(\overline{x},\overline{R},S)$. Then $\mathcal{A}\models \exists \; S \; \varphi(\overline{x},y,\overline{R},S)$.
\end{itemize}
\end{proof}

\section{Henkin models and faithfulness }\label{sec:henkinish}

Henkin models as defined in \S\ref{sec:intro} are required to interpret membership absolutely. It is natural to view them as models of many-sorted first-order logic, and hence to view $A$ as the interpretation of the sort associated to objects and $S_n[A]$ as the interpretation of the sort associated to $n$-ary second-order relations. In order to be able to apply the full metatheory of many-sorted first-order logic,\footnote{\cite[Chapter VI~ff]{Manzano1996-ff}, \cite[\S{4.3}]{Enderton2001-bm}, \cite[28 ff]{Marker2006-rd}, \cite[5]{Tent2012-jw}.} one needs (i) to have the freedom to temporarily relax the assumption that membership is interpreted absolutely, since e.g. the models coming from compactness will not in general have this property; and one needs simultaneously (ii) to know how to reinstate it by passing to a related model, since we want eventually to restate things in terms of Henkin models.

Hence we define:
\begin{defn} (Pre-Henkin model).

A \emph{pre-Henkin model} in a signature is a many-sorted structure of the form:
\begin{equation}\label{eqn:model}
\mathcal{A}=(A,S_1[A], S_2[A], \ldots; E_1^{\mathcal{A}}, E_2^{\mathcal{A}}, \ldots)
\end{equation}
where $A$ is a non-empty set, and where
\begin{enumerate}[leftmargin=*]
    \item \label{eqn:model:1} For all $n\geq 1$, the interpretation $E_n^{\mathcal{A}}$ of the membership symbol $E_n$ has sort $S_n[A]\times A^n$ and in the object-language we suppress $E_n$ and write $R\overline{a}$ instead of $E_n(R,a_1, \ldots, a_n)$.
    \item \label{eqn:model:2} For all $n\geq 1$ the structure $\mathcal{A}$ satisfies extensionality: 
    \begin{equation}
\forall \; R \; \forall \; S \; \bigg(R=S\leftrightarrow \forall \; \overline{x} \; (R\overline{x}\leftrightarrow S\overline{x})\bigg)
\end{equation}
    where $R, S$ are $n$-ary variables and the first-order tuple $\overline{x}$ has arity $n$.
    \item \label{eqn:model:3} The model $\mathcal{A}$ satisfies the comprehension schema~(\ref{eqn:compschema}) for first-order formulas in the signature.
    \item \label{eqn:model:4} For the interpretation of the signature, we have:
\begin{enumerate}[leftmargin=*]
    \item\label{eqn:model:4:a} For all individual constant symbols $c$ in the signature, it is interpreted as an element $c^{\mathcal{A}}$ of $A$.
    \item\label{eqn:model:4:b} For all $n\geq 1$ and all $n$-ary relation constant symbols $C$ in the signature, it is interpreted as an element $C^{\mathcal{A}}$ of $S_n[A]$.
    \item\label{eqn:model:4:c} For all $n\geq 1$ and all $n$-ary relation symbols $C$ in the signature, it is interpreted as a subset $C^{\mathcal{A}}$ of $A^n$.
\end{enumerate}    
\end{enumerate}
\end{defn}

\noindent In (\ref{eqn:model:1}), the idea is that $E_n$ interprets the membership relation between the domain $S_n[A]$ reserved for the $n$-ary second-order objects and $n$-tuples of elements from the first-order domain $A$. Since we only focus on models which satisfy first-order comprehension, and since the minimal set-theoretic apparatus described in \S\ref{sec:henkin} is all first-order definable, we may regard ourselves as working in the signature consisting just of membership symbols from (\ref{eqn:model:1}) and the symbols from (\ref{eqn:model:4}).

In \S\ref{sec:henkin} we only had symbols corresponding to (\ref{eqn:model:4:a}), (\ref{eqn:model:4:c}). It is useful to include (\ref{eqn:model:4:b}) since in the setting of many-sorted first-order logic, they will act differently under morphisms. For instance, if $f:\mathcal{A}\rightarrow \mathcal{B}$ is an embedding of pre-Henkin structures, then for a binary relation constant symbol $C_1$ in the sense of (\ref{eqn:model:4:b}), the standard definition of an embedding requires that (i)~$f(C^{\mathcal{A}}_1)=C^{\mathcal{B}}_1$, while for a binary relation symbol $C_2$ in the sense of (\ref{eqn:model:4:c}), the standard definition of an embedding requires that (ii)~$\mathcal{A}\models C_2 xy $ iff $\mathcal{B}\models C_2 f(x) f(y)$ for all $x,y$ in $A$. In the case where $\mathcal{A},\mathcal{B}$ are Henkin models and $f:\mathcal{A}\rightarrow \mathcal{B}$ is the inclusion map of a smaller Henkin model~$\mathcal{A}$ into a larger Henkin model $\mathcal{B}$, this is the difference between (i)~$C^{\mathcal{A}}_1=C^{\mathcal{B}}_1$ and (ii)~$C_2^{\mathcal{A}}\cap A^2= C_2^{\mathcal{B}}$. Case (ii) is what one wants when one considers e.g. a larger linear order $\mathcal{B}$ and its restriction to a smaller linear order~$\mathcal{A}$. But in model theory it is of course commonplace in proofs to name elements of the domains with fresh constant symbols, and for this one wants to be in case~(i). Since we are doing the proofs in these sections, we need to include (\ref{eqn:model:4:b}).

Note that in a signature with $n$-ary relation symbols $C$ as in (\ref{eqn:model:4:c}), there are three kinds of atomics: (I)~equalities, (II)~atomics coming from the membership symbols from (\ref{eqn:model:1}), and (III)~atomics coming from the $n$-ary constant relation symbols from~(\ref{eqn:model:4}). By extensionality (\ref{eqn:model:2}), all equalities on second-order entities are definable in terms of membership. By (\ref{eqn:model:3}), every $n$-ary constant relation symbol determines an element of $S_n[A]$, and so anything from~(III) is equivalent to something from~(II). Hence, in a fixed pre-Henkin model, as well as a fixed Henkin model, we can assume without loss of generality that the atomics are: (I$^{\prime}$)~equalities between individuals, (II$^{\prime}$)~atomics coming from the membership symbol.

Note that being a pre-Henkin model is preserved under isomorphism.

Having defined this more general notion, we can then specify the notion defined in \S\ref{sec:intro} as a special case:
\begin{defn}\label{defn:henkin} (Henkin models as certain kinds of pre-Henkin models).

A pre-Henkin model $\mathcal{A}$ is \emph{a Henkin model} if both of the following hold:
\begin{enumerate}[leftmargin=*]
    \item \label{defn:henkin:1} For each $n\geq 1$ one has $S_n[A]\subseteq P(A^n)$, where the latter denotes the powerset of the $n$-th Cartesian power of $A$.
    \item \label{defn:henkin:2} $E_n^{\mathcal{A}}$ is interpreted as set-theoretic membership in the metatheory, that is, $E_n^{\mathcal{A}}=\{(R, \overline{a})\in S_n[A]\times A^n: \overline{a}\in R\}$.
\end{enumerate}
\end{defn}

One has the following elementary observation:
\begin{prop}\label{prop:iso} (Relation between pre-Henkin models and Henkin models).

Every pre-Henkin model is isomorphic to a Henkin model.
\end{prop}
\begin{proof}
Suppose $\mathcal{A}$ a pre-Henkin model. Define model $\mathcal{B}$ by 
$B=A$ and $S_n[B]=\{ \{\overline{a}\in A^n: \mathcal{A}\models R\overline{a}\}: R\in S_n[A]\}$ and interpret membership as set-theoretic membership in the metatheory. Further, interpret $c^{\mathcal{B}}=c^{\mathcal{A}}$ for individual constant symbols $c$, and interpret $C^{\mathcal{B}} = \{\overline{a}\in A^n: \mathcal{A}\models C\overline{a}\}$ for $n$-ary constant relation symbols $C$ for $n\geq 1$ and for $n$-ary relation symbols $C$ for $n\geq 1$. Define bijection $f:\mathcal{A}\rightarrow \mathcal{B}$ by $f(a)=a$ for $a$ in $A$ and by $f(R)=\{\overline{a}\in A^n: \mathcal{A}\models R\overline{a}\}$ for $R$ in $S_n[A]$. 
Then $f:\mathcal{A}\rightarrow \mathcal{B}$ preserves the membership relation: $\mathcal{A}\models R\overline{a}$ iff $\overline{a}$ in $f(R)$, which happens iff $\mathcal{B}\models (f(R))(f(\overline{a}))$, by definition of $\mathcal{B}$'s absolute interpretation of membership and by $f:\mathcal{A}\rightarrow \mathcal{B}$ being the identity on $A^n$. Further, $f(c^{\mathcal{A}})=c^{\mathcal{B}}$ for individual constants $c$ since $f$ is the identity on $A$, and $\mathcal{A}\models C\overline{x}$ iff $\mathcal{B}\models C(f(\overline{x}))$ for $n$-ary relation symbols $C$ by definition of $C^{\mathcal{B}}$. Likewise for $n$-ary constant relation symbols $C$ one has by definition $f(C^{\mathcal{A}})=\{\overline{a}\in A^n: \mathcal{A}\models C\overline{a}\} = C^{\mathcal{B}}$. And $\mathcal{B}$ satisfies Definition~\ref{eqn:model}(\ref{eqn:model:2})-(\ref{eqn:model:3}) since it is isomorphic to $\mathcal{A}$. Finally, $\mathcal{B}$ satisfies, by construction, Definition~\ref{defn:henkin}(\ref{defn:henkin:1})-(\ref{defn:henkin:2}). 
\end{proof}
\noindent In the above proof, we use $f(\overline{x})$ as an abbreviation for the tuple $(f(x_1), \ldots, f(x_n))$. Similarly, often in what follows we write $f(\overline{x})=\overline{y}$, which is an abbreviation for $f(x_1)=y_1, \ldots, f(x_n)=y_n$.

However, not every substructure of a Henkin model is Henkin model:\footnote{For a similar point, see Boney \cite[168]{Boney2020-cw}. His solution is to continue to work with Henkin models, but to just dispense with substructures~\emph{per se} and to just work with the corresponding class of embeddings which preserve atomics in both directions.}
\begin{ex} (Among pre-Henkin models, being a Henkin model is not preserved under substructure).

Consider a standard model $\mathcal{B}$ in a countable signature where $B$ is uncountable. By Downward L\"owenheim-Skolem, take an elementary substructure 
\begin{equation}
\mathcal{A}=(A,S_1[A], S_2[A], \ldots; E_1^{\mathcal{A}}, E_2^{\mathcal{A}}, \ldots)
\end{equation}
where $A$ is countably infinite and each $S_n[A]$ is countably infinite. Then $\mathcal{A}$ is a pre-Henkin model. Let $X$ be the unique element of $S_1[A]$ such that $\mathcal{A}\models \forall x \; Xx$. Then since it is an elementary substructure one has that $\mathcal{B}\models \forall \; x \; Xx$. Then since $\mathcal{B}$ is a Henkin model, $X$ is in $S_1[B]$ and by extensionality $X=B$. Since $A$ is countable and $B$ is uncountable, we have that $X$ is in $S_1[A]\setminus P(A)$, and thus $S_1[A]\nsubseteq P(A)$. Thus $\mathcal{A}$ is not a Henkin model. 
\end{ex}

We can also describe this example in terms of embeddings:
\begin{ex} (Among pre-Henkin models, being a Henkin model is not preserved under taking images of embeddings).

Let $\mathcal{A}, \mathcal{B}$ be as in the previous example. By Proposition~\ref{prop:iso} there is a Henkin model $\mathcal{A}_0$ and isomorphism $f:\mathcal{A}_0\rightarrow \mathcal{A}$. Since $\mathrm{id}_{\mathcal{A}}:\mathcal{A}\rightarrow \mathcal{B}$ is an elementary embedding, the composition $f=\mathrm{id}_{\mathcal{A}}\circ f: \mathcal{A}_0\rightarrow \mathcal{B}$ is an elementary embedding, and its image is $\mathcal{A}$. Hence the image of an elementary embedding between Henkin models is not necessarily a Henkin model.
\end{ex}

This example motivates the following definition, which is the natural analogue of Marker \cite[Definition 1.2]{Marker1984aa} in the setting of second-order logic:
\begin{defn} (Faithful embeddings).

Suppose that $\mathcal{A},\mathcal{B}$ are two pre-Henkin models. Then $f:\mathcal{A}\rightarrow \mathcal{B}$ is a \emph{faithful embedding} if it is an embedding of $\mathcal{A}$ into $\mathcal{B}$ and for all $n\geq 1$ and all $R$ in $S_n[A]$ and all $R^{\prime}$ in $S_n[B]$ with $f(R)=R^{\prime}$ and all $\overline{b}$ in $B^n$ with $\mathcal{B}\models R^{\prime}\overline{b}$, there is $\overline{a}$ in $A^n$ with $f(\overline{a})=\overline{b}$.
\end{defn}

These are related by the following:
\begin{prop}\label{prop:imagesoffaith} (Henkin models and images of faithful embeddings).

Suppose that $\mathcal{B}$ is a Henkin model. Suppose that $\mathcal{D}$ is a  substructure of $\mathcal{B}$. Then the following are equivalent:
\begin{enumerate}[leftmargin=*]
    \item \label{faithful:1} $\mathcal{D}$ is a Henkin model.
    \item \label{faithful:2} There is a Henkin model $\mathcal{A}$ and a faithful embedding $f:\mathcal{A}\rightarrow \mathcal{B}$ with image $\mathcal{D}$.
\end{enumerate}
\end{prop}
\begin{proof}
Suppose (\ref{faithful:1}); we show (\ref{faithful:2}). Then the identity map $\mathrm{id}:\mathcal{D}\rightarrow \mathcal{B}$ is a faithful embedding. For, suppose that $n\geq 1$ and $R$ in $S_n[D]$ and $\overline{b}$ in $B^n$ with $\mathcal{B}\models R\overline{b}$. Then $\overline{b}$ in $R$ since $\mathcal{B}$ is a Henkin model, and $R\subseteq D^n$ since $\mathcal{D}$ is a Henkin model. Then $\overline{b}$ in $D^n$, which is what we wanted to show.

Suppose (\ref{faithful:2}); we show (\ref{faithful:1}). Suppose $\mathcal{A}$ is a Henkin model and $f:\mathcal{A}\rightarrow \mathcal{B}$ is a faithful embedding with image $\mathcal{D}$. Then $D=\{f(a): a\in A\}$ and $S_n[D]=\{f(R): R\in S_n[A]\}$ and $E^{\mathcal{D}}_n = \{ (f(R), f(\overline{a})): (R,\overline{a})\in S_n[A]\times A^n \; \& \; \mathcal{A}\models R\overline{a}\}$ and $c^{\mathcal{D}}=f(c^{\mathcal{A}})$ for individual constant symbols $c$, and $C^{\mathcal{D}}=f(C^{\mathcal{A}})$ for $n$-ary relation constant symbols $C$, and $C^{\mathcal{D}} = \{f(\overline{x}): \mathcal{A}\models C\overline{x}\}$ for $n$-ary relation symbols $C$.
\begin{enumerate}[leftmargin=*, label=(\alph*), ref=\alph*]
    \item \label{prop:imagesoffaith:a} We first show that $S_n[D]\subseteq P(D^n)$. Suppose that $R^{\prime}$ in $S_n[D]$; we must show that $R^{\prime}\subseteq D^n$. Since $R^{\prime}$ in $S_n[D]$, we have $R^{\prime}=f(R)$ for some $R$ in $S_n[A]$. Since $f:\mathcal{A}\rightarrow \mathcal{B}$ preserves sorts, one has that $R^{\prime}$ in $S_n[B]$, and since $\mathcal{B}$ is a Henkin model we have that $R^{\prime}\subseteq B^n$. Suppose that $\overline{b}$ in $R^{\prime}$. Since $\mathcal{B}$ is a Henkin model we have $\mathcal{B}\models R^{\prime}\overline{b}$. Since $f:\mathcal{A}\rightarrow \mathcal{B}$ is faithful, there is $\overline{a}$ in $A^n$ with $f(\overline{a})=\overline{b}$. Then $\overline{b}$ in $D^n$ by definition of $D$.
\item\label{prop:imagesoffaith:b} We second show that $E^{\mathcal{D}}_n$ is equal to $\{(R^{\prime}, \overline{b})\in S_n[D]\times D^n : \overline{b}\in R^{\prime}\}$:
\begin{itemize}[leftmargin=*]
    \item Suppose $(f(R), f(\overline{a}))$ in $E^{\mathcal{D}}_n$. Then since $\mathcal{A}\models R\overline{a}$ and $f:\mathcal{A}\rightarrow \mathcal{B}$ is a homomorphism we have $\mathcal{B}\models (f(R))(f(\overline{a}))$. Since $\mathcal{B}$ is a Henkin model we have $f(\overline{a})$ in $f(R)$.
    \item Suppose that $(R^{\prime}, \overline{b})$ in $S_n[D]\times D^n$ with $\overline{b}$ in $R^{\prime}$. Then $R^{\prime}=f(R)$ for some $R$ in $S_n[A]$ and $\overline{b}=f(\overline{a})$ for some $\overline{a}$ in $A^n$. Since $\mathcal{B}$ is a Henkin model $\mathcal{B}\models R^{\prime}\overline{b}$. Since $f:\mathcal{A}\rightarrow \mathcal{B}$ is an embedding we have $\mathcal{A}\models R\overline{a}$. This then puts $(f(R), f(\overline{a}))$ in $E^{\mathcal{D}}_n$, and so puts $(R^{\prime},\overline{b})$ in $E^{\mathcal{D}}_n$.
\end{itemize}
    \item \label{prop:imagesoffaith:3} Finally, $\mathcal{D}$ satisfies the axioms in Definition~\ref{eqn:model}(\ref{eqn:model:2})-(\ref{eqn:model:3}) since $f:\mathcal{A}\rightarrow \mathcal{D}$ is an isomorphism and $\mathcal{A}$ satisfies these axioms.
\end{enumerate}

\end{proof}

\section{Proof of Theorems~\ref{thm:fefermanesq}, \ref{thm:fefermanesqfour}}\label{sec:proofmain}

\thmfefermanesq*

As mentioned in the abstract and the introduction, Theorem~\ref{thm:fefermanesq} basically follows from Feferman's Preservation Theorem \cite{Feferman1968-fi}, which Marker \cite{Marker1984aa} proved model-theoretically. Since readers interested in generalized quantifiers may not be interested in the model-theoretic details of Marker's presentation, we specialize Marker's proof to the setting of second-order logic. Further, since the proof is short, this requires about the same amount of space needed to verify that Theorem~\ref{thm:fefermanesq} follows from from Marker's result.\footnote{In outline, here is how one would verify that Theorem~\ref{thm:fefermanesq} follows from Marker's result. First, one defines $u\triangleleft_n R$ iff $R$ is $n$-ary and $u$ is an element of the field $\mathscr{U}R$ of $R$. Then one defines the class of formulas $\Sigma_n$ from $\triangleleft_n$ just like Marker defines his class $\Sigma$ from $\triangleleft$ (\cite[Definition 1.1]{Marker1984aa}). Then, one defines the class of formulas $\Sigma$ to be the union of the $\Sigma_n$. Finally, one has to go through Marker's proof and verify that his results about a single $\triangleleft$ work when one considers a countable sequence $\triangleleft_n$. This last step amounts to reproving Marker's theorem for this generalization. Lastly, one would show that the $\Sigma$ formulas align extensionally with the $\Sigma^1_1$ formulas in models of $\mathsf{\Sigma^1_1\mbox{-}{AC}_0^{-}}$.}

\begin{proof}

For (\ref{thm:fefermanesq:1}), suppose that $T$ is a theory extending $\mathsf{\Sigma^1_1\mbox{-}{AC}_0}$, and suppose that $Q$ is Henkin upward independent relative to $T$.  We show that $Q$ is $\Sigma^1_1$-definable relative to $T$. Without loss of generality, $Q$ is instantiated in a model of $T$ (otherwise choose an inconsistent $\Sigma^1_1$-definition).

Let $R_1, \ldots, R_{\ell}$ be new relation constant symbols, in the sense of Definition~\ref{eqn:model}(\ref{eqn:model:4:b}), corresponding to the type of $Q$. We abbreviate $R_1, \ldots, R_{\ell}$ as $\overline{R}$. It suffices to show: there is a $\Sigma^1_1$-formula $\Phi$ such that the theory $T$ proves: $\Phi(\overline{R})\leftrightarrow Q(\overline{R})$.

Note that since $Q$ is instantiated in a model of $T$, we have that there is a model $\mathcal{A}$ of $T+Q(\overline{R})$. For any model $\mathcal{A}$ in the expanded signature, let~$\Sigma^1_1(\mathcal{A})$ be the set of all~$\Sigma^1_1$-sentences in the expanded signature which are true in~$\mathcal{A}$. Consider the following:
\begin{case}\label{eqn:markerpair} There is a pair of models $\mathcal{A}_0$ and $\mathcal{B}_0$  such that $\mathcal{A}_0\models T+Q(\overline{R})$ and $\mathcal{B}_0\models T+\neg Q(\overline{R}) +\Sigma^1_1(\mathcal{A}_0)$.
\end{case}

Suppose that Case~\ref{eqn:markerpair} holds; we show that this eventuates in a violation of Henkin upward independence. Since $\mathcal{B}_0\models \Sigma^1_1(\mathcal{A}_0)$, one has that $\mathcal{B}_0$ is infinite iff $\mathcal{A}_0$ is infinite; and likewise $\mathcal{B}_0$ is finite and of size $n$ iff $\mathcal{A}_0$ is finite and of size~$n$. Hence, suppose that both $\mathcal{A}_0, \mathcal{B}_0$ are finite and of size~$n$. Finite models of the same size which satisfy the same first-order sentences in a signature are isomorphic and hence satisfy the same second-order sentences, which is impossible since $\mathcal{A}_0\models Q(\overline{R})$ and $\mathcal{B}_0\models \neg Q(\overline{R})$. Henceforth suppose that both $\mathcal{A}_0, \mathcal{B}_0$ are infinite, and without loss of generality countably infinite. Let $(\mathcal{A}_0, \mathcal{B}_0)$ be the disjoint sum of the two many-sorted structures. Let $(\mathcal{A}, \mathcal{B})$ be a countable recursively saturated elementary extension of $(\mathcal{A}_0, \mathcal{B}_0)$.\footnote{E.g. \cite[169]{Marker2006-rd}, \cite[379]{Simpson2009-ra}.} By taking an isomorphic copy \emph{\`a la} Proposition~\ref{prop:iso}, we may assume that both $\mathcal{A},\mathcal{B}$ are Henkin models. Enumerate both the first and second-order part of $\mathcal{A}$ simultaneously as $\alpha_1, \alpha_2, \ldots$, and similarly enumerate $\mathcal{B}$ as $\beta_1, \beta_2, \ldots$. We build a faithful embedding $f=\bigcup_n f_n:\mathcal{A}\rightarrow \mathcal{B}$ in stages $f_0,f_1, f_2, \ldots$ such that $f_n$ satisfies: 
\begin{enumerate}[leftmargin=*, label=(\roman*), ref=\roman*]
\item\label{eqn:constmarker1} The domain of $f_n$ is a finite subset including $\alpha_1, \ldots, \alpha_n$.
\item\label{eqn:constmarker2} For any $\Sigma^1_1$-formula $\varphi(\upsilon_1, \ldots, \upsilon_j)$ with free variables among $\upsilon_1, \ldots, \upsilon_j$, if $\mathcal{A}\models \varphi(\gamma_1, \ldots, \gamma_j)$ where $\gamma_1, \ldots, \gamma_j$ enumerates the domain of $f_n$, then $\mathcal{B}\models \varphi(f_n(\gamma_1), \ldots, f_n(\gamma_j))$.
\item\label{eqn:constmaker3} If $\beta_i$ is the least first-order object with $i\leq n$ which is in some tuple in some second-order object $\beta$ in the range of $f_n$, then $\beta_i$ is in the range of~$f_n$.
\end{enumerate}
Note that while each of the variables $\upsilon_1, \ldots, \upsilon_j$ in (\ref{eqn:constmarker2}) is either first-order or second-order (and if second-order is unary, binary, etc.), we are enumerating them all simultaneously but not marking their status in the notation: one could easily mark them as first-order or second-order (and if second order one could mark whether they are unary, binary, etc.), but this would terribly clutter the notation. By way of motivation for these conditions, note that:
\begin{itemize}[leftmargin=*]
    \item Condition~(\ref{eqn:constmarker2}) ensures that $f:\mathcal{A}\rightarrow \mathcal{B}$ is an embedding. In particular it ensures the injectivity of $f$ and and that $f$ preserves the individual constant symbols and $n$-ary relation constant symbols. For, consider the formula $\varphi_c(\upsilon_k)\equiv c=\upsilon_k$ for an individual constant symbol $c$ and individual variable $\upsilon_k$, and consider the formula $\varphi_i(\upsilon_j)\equiv R_i = \upsilon_j$ for one of the $R_i$ from our sequence of $R_1, \ldots, R_{\ell}$ of relation constants symbols and $\upsilon_j$ an $n$-ary variable. Since $\mathcal{A}\models \varphi_c(c^{\mathcal{A}})$ we have by condition~(\ref{eqn:constmarker2}) that $\mathcal{B}\models \varphi_c(f(c^{\mathcal{A}}))$, which is to say that $f(c^{\mathcal{A}})=c^{\mathcal{B}}$. Since $\mathcal{A}\models \varphi_i(R_i^{\mathcal{A}})$ we have by condition~(\ref{eqn:constmarker2}) that $\mathcal{B}\models \varphi_i(f(R^{\mathcal{A}}_i))$, which is to say that $f(R^{\mathcal{A}}_i)=R^{\mathcal{B}}_i$.
    \item Condition~(\ref{eqn:constmaker3}) ensures that $f:\mathcal{A}\rightarrow \mathcal{B}$ is a faithful embedding.
\end{itemize}
For the base case, one takes $f_0=\emptyset$, and condition~(\ref{eqn:constmarker2}) is satisfied because $\mathcal{B}\models \Sigma^1_1(\mathcal{A})$. For the inductive step, we consider the following recursive type $\Gamma(\upsilon)$, where $\gamma_1, \ldots, \gamma_j$ enumerates the domain of $f_n$. We insert an initial formula in  $\Gamma(\upsilon)$ which says that $\upsilon$ comes from $\mathcal{B}$, and in general we put a formula in $\Gamma(\upsilon)$ if it has the form 
\begin{equation}\label{eqn:theofrm}
\varphi^{\mathcal{A}}(\gamma_1, \ldots, \gamma_j, \alpha_{n+1})\rightarrow \varphi^{\mathcal{B}}(f_n(\gamma_1), \ldots, f_n(\gamma_j), \upsilon)
\end{equation}
where $\varphi(\upsilon_1, \ldots, \upsilon_j, \upsilon)$ is a $\Sigma^1_1$-formula, and where as in (\ref{eqn:constmarker2}) the variables $\upsilon_1, \ldots, \upsilon_j, \upsilon$ may be first-order or second-order. This recursive type $\Gamma(\upsilon)$ is indeed finitely satisfied. For, suppose that we have $m$ formulas:
\begin{equation}
\varphi_1(\upsilon_1, \ldots, \upsilon_j, \upsilon), \ldots, \varphi_m(\upsilon_1, \ldots, \upsilon_j, \upsilon)
\end{equation}
By reordering, we may suppose that $\mathcal{A}\models \varphi_i(\gamma_1, \ldots, \gamma_j, \alpha_{n+1})$ iff $i\leq m_0$, for some $m_0\leq m$. Then $\mathcal{A}\models \exists \; \upsilon \; \bigwedge_{i=1}^{m_0} \varphi_i(\gamma_1, \ldots, \gamma_j, \upsilon)$. Since this is itself a $\Sigma^1_1$-formula, by (\ref{eqn:constmarker2}) one has that $\mathcal{B}\models \exists \; \upsilon \; \bigwedge_{i=1}^{m_0} \varphi_i(f_n(\gamma_1), \ldots, f_n(\gamma_j), \upsilon)$. Thus $\Gamma(\upsilon)$ is indeed finitely satisfied and so satisfied in $\mathcal{B}$ by some $\beta$, and then we set $f_{n+1}(\alpha_{n+1})=\beta$. Thus we have satisfied~(\ref{eqn:constmarker1})-(\ref{eqn:constmarker2}). To ensure~(\ref{eqn:constmaker3}), suppose that $\beta_i$ is the least first-order object with $i\leq n+1$ which is in some tuple in some second-order object $\beta$ in the range of $f_{n+1}$, say $\beta=f(\gamma_p)$. Without loss of generality the tuple is, say, a 2-tuple, so that $(\beta_u,\beta_i)\in \beta$. Consider the following recursive type $\Gamma^{\ast}(v)$, where $v$ is a first-order variable.  We put an initial formula in  $\Gamma^{\ast}(v)$ which says that $v$ comes from $\mathcal{A}$, and in general we put a formula in $\Gamma^{\ast}(v)$ if it has the form 
\begin{equation}
\varphi^{\mathcal{B}}(f_n(\gamma_1), \ldots, f_n(\gamma_j), f_{n+1}(\alpha_{n+1}), \beta_i)\rightarrow \varphi^{\mathcal{A}}(\gamma_1, \ldots, \gamma_j, \alpha_{n+1}, v)
\end{equation}
where $\varphi(\upsilon_1, \ldots, \upsilon_j, \upsilon_{j+1}, v)$ is a $\Pi^1_1$-formula. This recursive type is indeed finitely satisfied. For, suppose that we have $m$-such formulas 
\begin{equation}
\varphi_1(\upsilon_1, \ldots, \upsilon_j, \upsilon_{j+1}, v), \ldots, \varphi_m(\upsilon_1, \ldots, \upsilon_j, \upsilon_{j+1}, v)
\end{equation}
By reordering, we again may suppose that $\mathcal{B}\models \varphi_i(f_n(\gamma_1), \ldots, f_n(\gamma_j), f_{n+1}(\alpha_{n+1}), \beta_i)$ iff $i\leq m_0$, for some $m_0\leq m$. Then one has:
\begin{equation}
\mathcal{B}\models \exists \; v \; \exists \; u \;  (f(\gamma_p))(u,v) \; \wedge\; \bigwedge_{i=1}^{m_0} \varphi_i(f_n(\gamma_1), \ldots, f_n(\gamma_j), f_{n+1}(\alpha_{n+1}), v)
\end{equation}
Since $\mathcal{B}$ is a model of $\mathsf{\Sigma^1_1\mbox{-}AC_0^{-}}$, this is equivalent to a $\Pi^1_1$-formula over $\mathcal{B}$, and similarly over $\mathcal{A}$, and thus by~(\ref{eqn:constmarker2}) holding for $f_{n+1}$ as defined so far, we have that 
\begin{equation}
\mathcal{A}\models \exists \; v \; \exists \; u \;  \gamma_p(u,v)\;\wedge \; \bigwedge_{i=1}^{m_0} \varphi_i(\gamma_1, \ldots, \gamma_j, \alpha_{n+1}, v)
\end{equation}
Thus $\Gamma^{\ast}(v)$ is indeed finitely satisfied and so satisfied in $\mathcal{A}$ by some $\gamma$, and then we set $f_{n+1}(\gamma)=\beta_i$. This completes the construction of the faithful embedding $f:\mathcal{A}\rightarrow \mathcal{B}$. Let $\mathcal{D}$ be the image of $\mathcal{A}$ under the faithful embedding $f:\mathcal{A}\rightarrow \mathcal{B}$. By Proposition~\ref{prop:imagesoffaith}, we have that $\mathcal{D}$ is a Henkin model which is a substructure of $\mathcal{B}$. But the assumption that we are in Case~\ref{eqn:markerpair} implies that $\mathcal{A}\models Q(\overline{R})$, and since $f:\mathcal{A}\rightarrow \mathcal{D}$ is an isomorphism, we have $\mathcal{D}\models Q(\overline{R})$. By Henkin upwards independence, we have $\mathcal{B}\models Q(\overline{R})$, which contradicts that we are in Case~\ref{eqn:markerpair}, which says that $\mathcal{B}\models \neg Q(\overline{R})$.

Hence, rather we are in the following case

\begin{case}\label{eqn:markerpair2} There is no pair of models $\mathcal{A}_0$ and $\mathcal{B}_0$  such that $\mathcal{A}_0\models T+Q(\overline{R})$ and $\mathcal{B}_0\models T+\neg Q(\overline{R}) +\Sigma^1_1(\mathcal{A}_0)$.
\end{case}

Then for any model $\mathcal{A}\models T+Q(\overline{R})$, one has that $T+\neg Q(\overline{R}) +\Sigma^1_1(\mathcal{A})$ is inconsistent and so by compactness there is a $\Sigma^1_1$-sentence $\Phi_{\mathcal{A}}$ in $\Sigma^1_1(\mathcal{A})$ such that 
\begin{equation}\label{eqn:onehalfMarkerlem2.1}
T\vdash \Phi_{\mathcal{A}} \rightarrow Q(\overline{R})
\end{equation}
In this, we appeal to the fact that $\Sigma^1_1$-formulas are closed under finite conjunctions. Then consider the theory 
\begin{equation}
\Delta = \{\neg \Phi_{\mathcal{A}}: \mathcal{A} \models T+ Q(\overline{R})\}
\end{equation}
Then $T+Q(\overline{R})+\Delta$ is inconsistent. For, suppose that it had a model $\mathcal{A}$. Then it would model $\neg \Phi_{\mathcal{A}}$ (since this sentence is in $\Delta$) and it would model $\Phi_{\mathcal{A}}$ (since this sentence is, by construction, in $\Sigma^1_1(\mathcal{A})$, and hence true on $\mathcal{A}$). Since $T+Q(\overline{R})+\Delta$ is inconsistent, there are a finite number of models $\mathcal{A}_1, \ldots, \mathcal{A}_m \models T+ Q(\overline{R})$ such that $T + Q(\overline{R})$ proves $\bigvee_{i=1}^m \Phi_{\mathcal{A}_i}$. But then together with (\ref{eqn:onehalfMarkerlem2.1}), we have that
\begin{equation}
T\vdash  Q(\overline{R}) \leftrightarrow \bigvee_{i=1}^m \Phi_{\mathcal{A}_i}
\end{equation}
This finishes the proof of (\ref{thm:fefermanesq:1}) 

The proof of (\ref{thm:fefermanesq:2}) follows from (\ref{thm:fefermanesq:1}) by replacing $Q$ by $\neg Q$; and (\ref{thm:fefermanesq:3}) follows directly from (\ref{thm:fefermanesq:1})-(\ref{thm:fefermanesq:2}).

\end{proof}

Now we turn to the proof of Theorem~\ref{thm:fefermanesqfour}. We first prove a preliminary theorem:

\begin{thm}\label{thm:fefermanesqfourpre} (Sufficient conditions for grades of bounded definability).

Suppose that $T$ is a theory extending $\mathsf{ACA}_0^{-}$ in a signature with only finitely many individual constants which is closed downwards under internal substructures. 

Suppose that $Q$ is a generalized quantifier which is $\Sigma^1_1$-definable (resp. $\Pi^1_1$-definable, resp. first-order definable) relative to $T$ and Henkin conservative relative to $T$ and has strong existential import relative to $T$. Then~$Q$ is bounded $\Sigma^1_1$-definable (resp. bounded $\Pi^1_1$-definable, resp. bounded first-order definable) relative to $T$.
\end{thm}

\begin{proof}

Consider a Henkin model $\mathcal{B}$ of $T$, and consider the $\Sigma^1_1$-definition $\Psi(\overline{R},\overline{R}^{\prime})\equiv \exists \; \overline{S} \; \Psi_0(\overline{R},\overline{R}^{\prime}, \overline{S})$ of $Q(\overline{R},\overline{R}^{\prime})$, where $\Psi_0(\overline{R},\overline{R}^{\prime}, \overline{S})$ contains no second-order quantifiers. Then the following $\Sigma^1_1$-formula $\Phi(\overline{R},\overline{R}^{\prime})$ also defines $Q(\overline{R},\overline{R}^{\prime})$ relative to $T$, where we have added the dummy existentials, where $c_1, \ldots, c_n$ are the individual constants in the signature:
\begin{equation}
\exists \; \overline{S} \; \bigg(\big((\exists \; x \; x=x) \wedge \bigwedge_{i=1}^n (\exists \; x \; x=c_i)\big) \wedge \Psi_0(\overline{R},\overline{R}^{\prime}, \overline{S})\bigg)
\end{equation}
Let $\Phi^{\mathscr{U}\overline{R}}(\overline{R},\overline{R}^{\prime})$ be the result of bounding all quantifiers in $\Phi(\overline{R},\overline{R}^{\prime})$ to $\mathscr{U}\overline{R}$, as in Definition~\ref{defn:binding}(\ref{defn:binding3}), so that $\Phi^{\mathscr{U}\overline{R}}(\overline{R},\overline{R}^{\prime})$ is equivalent to:
\begin{equation}\label{eqn:iamlongeqn}
\exists \; \overline{S}\subseteq \mathscr{U}\overline{R} \; \bigg(\big((\exists \; x \; (\mathscr{U}\overline{R})x) \wedge \bigwedge_{i=1}^n (\mathscr{U}\overline{R})c_i \big) \wedge \Psi_0^{\mathscr{U}\overline{R}}(\overline{R},\overline{R}^{\prime}, \overline{S})\bigg)
\end{equation}

Let $\overline{R},\overline{R}^{\prime}$ be fixed in $S_{\overline{n}}[B]$. We claim that $\mathcal{B}\models Q(\overline{R},\overline{R}^{\prime})\leftrightarrow \Phi^{\mathscr{U}\overline{R}}(\overline{R},\overline{R}^{\prime}\cap \mathscr{U}\overline{R})$.

First suppose that $\mathscr{U} \overline{R}$ is empty or does not include the interpretation of one of the individual constants $c_1, \ldots, c_n$. By strong existential import (cf. Definition~\ref{defn:strongexistential}), one has that $\mathcal{B}\models \neg Q(\overline{R},\overline{R}^{\prime})$. By (\ref{eqn:iamlongeqn}) we have $\mathcal{B}\models \neg \Phi^{\mathscr{U}\overline{R}}(\overline{R},\overline{R}^{\prime})$. Hence we have the biconditional since both sides are false.

Second suppose that $\mathscr{U} \overline{R}$ is non-empty and contains all of the interpretations of the individual constants $c_1, \ldots, c_n$. Then by choosing $A=\mathscr{U} \overline{R}$ we have a Henkin internal substructure $\mathcal{A}$ of $\mathcal{B}$ (cf. Definition~\ref{defn:internal}). Since $T$ is closed downwards under internal substructures, we have that $\mathcal{A}\models T$. Further, one has $\mathcal{B}\models Q(\overline{R},\overline{R}^{\prime})$ iff $\mathcal{B}\models Q(\overline{R},\overline{R}^{\prime}\cap \mathscr{U}\overline{R})$ by Henkin conservativity. By Henkin domain invariance relative to~$T$ together with Proposition~\ref{prop:absolutenboolean}, this happens iff $\mathcal{A}\models  Q(\overline{R},\overline{R}^{\prime}\cap \mathscr{U}\overline{R})$. Using the $\Sigma^1_1$-definition in $\mathcal{A}$ this happens iff $\mathcal{A}\models \Phi(\overline{R},\overline{R}^{\prime}\cap \mathscr{U}\overline{R})$. By Theorem~\ref{thm:internal}, this happens iff $\mathcal{B}\models \Phi^{\mathscr{U}\overline{R}}(\overline{R},\overline{R}^{\prime}\cap \mathscr{U}\overline{R})$. Hence indeed $\mathcal{B}\models Q(\overline{R},\overline{R}^{\prime})\leftrightarrow \Phi^{\mathscr{U}\overline{R}}(\overline{R},\overline{R}^{\prime}\cap \mathscr{U}\overline{R})$. 

Thus $Q(\overline{R},\overline{R}^{\prime})$ is bounded $\Sigma^1_1$-definable. 

The same argument works for the $\Pi^1_1$-case and the first-order case.

\end{proof}

Here is then the theorem on bounded definability:

\thmfefermanesqfour*

\begin{proof}
This follows from Theorem~\ref{thm:fefermanesq} and Theorem~\ref{thm:fefermanesqfourpre}.
\end{proof}

\section{Proof of Theorems~\ref{thm:fodefinable}-\ref{thm:fodefinabletwo}}\label{sec:fodefinable}

First we define a substitution notion:
\begin{defn}\label{defn:mysubstitution} (Substituting a first-order formula for a relation variable).

Suppose that the type of $\overline{R}$ is $\overline{n}=( n_1, \ldots, n_k)$. Suppose that $\Phi(\overline{R})$ is a formula of second-order logic which has no second-order quantifiers and only second-order free variables, including the displayed $\overline{R}$.

Suppose that we have specified an additional first-order signature, which has relation symbols corresponding to all the free second-order variables of $\Phi(\overline{R})$ \emph{excluding} those of the displayed $\overline{R}$ (our substitution notion will remove these displayed $\overline{R}$).

Then a tuple of first-order formulas $\overline{\xi}$ in this signature is of \emph{type $\overline{n}$} if it has the form 
\begin{equation}\label{eqn:displayallvariables}
(\xi_1(\overline{v_1},\overline{x}), \ldots, \xi_k(\overline{v_k}, \overline{x}))
\end{equation}
where the displayed $\overline{v_i}$ are of length $n_i$, and all the variables in $\overline{v_1}, \ldots, \overline{v_n}, \overline{x}$ are pairwise distinct, and do not occur free or bound in $\Phi(\overline{R})$.

If $\overline{\xi}$ is of type $\overline{n}$, then we let $\Phi(\overline{\xi}(\cdot, \overline{x}))$ be the first-order formula in the signature which is the result of replacing each instance of $R_{i}\overline{u}$ in $\Phi(\overline{R})$ with $\xi_i(\overline{u}, \overline{x})$. 

\end{defn}

Hence $\Phi(\overline{\xi}(\cdot,\overline{x}))$ has only the free variables $\overline{x}$ coming from $\overline{\xi}$. Note that when we are writing $\Phi(\overline{\xi}(\cdot, \overline{x}))$, we are suppressing the $\overline{v}$ variables in (\ref{eqn:displayallvariables}). If one wanted a more explicit notation, one could write $\Phi(\overline{\xi}(\overline{v},\overline{x}))$, but this has the disadvantage of creating the impression that $\overline{v}$ is free in this formula, which it is not. The following elementary example is illustrative:

\begin{ex} (Example of substitution).

Let $\overline{R}$ be $R_1,R_2$, where $R_1$ is unary and $R_2$ is binary, and let $\Phi(\overline{R})$ be the following:
\begin{equation}
\forall \; x \; \forall \; y \; \big((R_1x\wedge R_1y \wedge x\neq y)\rightarrow R_2xy\big)
\end{equation}
Then the type of $\overline{R}$ is $( 1,2)$. Suppose that we specify the first-order signature with unary relation symbol $F$ and binary relation symbol $S$. Consider the following tuple of first-order formulas $\overline{\xi}$ of type $( 1,2)$:
\begin{equation}
\big( Fv_{1,1} \wedge Sx_1 v_{1,1},\hspace{3mm} S v_{2,2} v_{2,1} \rightarrow S v_{2,1} x_2\big) 
\end{equation}
Then one has that $\Phi(\overline{\xi}(\cdot, \overline{x}))$ is 
\begin{equation}
\forall \; x \; \forall \; y \; \big((Fx \wedge Sx_1 x\wedge Fy \wedge Sx_1 y \wedge x\neq y)\rightarrow (Syx\rightarrow Sxx_2)\big)
\end{equation}
which has only free variables $x_1, x_2$.
\end{ex}

We begin the march to the proof of Theorems~\ref{thm:fodefinable}-\ref{thm:fodefinabletwo} by showing the following.

\begin{thm}\label{thm:fromdeltatoaca} (From $\Delta^1_1$-definable to first-order definable in $\mathsf{ACA_0^{wo}}$).

 Suppose that $Q(\overline{R})$ is a generalized quantifier which is $\Delta^1_1$-definable relative to $\mathsf{ACA_0^{wo}}$. Then $Q(\overline{R})$ is first-order definable relative to $\mathsf{ACA_0^{wo}}$.
 
\end{thm}
The same result holds for $\mathsf{ACA_0^{-}}$. We prove it for $\mathsf{ACA_0^{wo}}$ because we need this for a later result in this section.
\begin{proof}

Suppose that $Q(\overline{R})$ is definable in $\mathsf{ACA_0^{wo}}$ by a $\Sigma^1_1$-formula $\exists \; \overline{S} \; \Phi(\overline{R},\overline{S})$ where $\Phi(\overline{R},\overline{S})$ has no second-order quantifiers and has only the displayed variables free, which are all second-order. Suppose further that $Q(\overline{R})$ is definable in $\mathsf{ACA_0^{wo}}$ by a $\Pi^1_1$-formula $\forall \; \overline{S^{\prime}} \; \Psi(\overline{R},\overline{S^{\prime}})$ where $\Psi(\overline{R},\overline{S^{\prime}})$ has no second-order quantifiers and has only the displayed variables free, which are all second-order. Suppose that $\overline{R}$ is of type $\overline{n}$, and $\overline{S}$ is of type $\overline{n^{\prime}}$, and $\overline{S^{\prime}}$ is of type $\overline{n^{\prime\prime}}$. 

Let $T_0$ be the first-order theory which schematically says that $\preceq$ is a well-order, and whose signature includes relation symbols corresponding to $\overline{R}$.

If a tuple of first-order formulas $\overline{\xi}$ in this signature is of type $\overline{n^{\prime}}$, then we form $\Phi(\overline{R},\overline{\xi}(\cdot, \overline{x}))$ from $\Phi(\overline{R},\overline{S})$ as in Definition~\ref{defn:mysubstitution}. Likewise, if a tuple of first-order formulas $\overline{\chi}$ in this signature is of type $\overline{n^{\prime\prime}}$, then we form $\Psi(\overline{R},\overline{\chi}(\cdot,\overline{y}))$ from $\Phi(\overline{R},\overline{S^{\prime}})$ as in Definition~\ref{defn:mysubstitution}.

Let $T$ be $T_0$ plus the set of all first-order sentences $\forall \; \overline{y} \; \Psi(\overline{R}, \overline{\chi}(\cdot, \overline{y}))$, as $\overline{\chi}$ ranges over tuples of first-order formulas in the signature of type $\overline{n^{\prime\prime}}$. Then we claim that there is a finite number of tuples of first-order formulas $\overline{\xi_1}, \ldots, \overline{\xi_n}$ in the signature of type $\overline{n^{\prime}}$ such that 
\begin{equation}\label{eqn:whattproves}
T\vdash \bigvee_{i=1}^n \exists \; \overline{x} \; \Phi(\overline{R}, \overline{\xi_i}(\cdot, \overline{x}))
\end{equation}
Suppose not. Then by compactness there is a first-order model $\mathcal{M}$ such that
\begin{equation}
\mathcal{M}\models T_0 \wedge \forall \; \overline{y} \; \Psi(\overline{R}, \overline{\chi}(\cdot, \overline{y})) \wedge \forall \; \overline{x} \; \neg \Phi(\overline{R}, \overline{\xi}(\cdot, \overline{x}))
\end{equation}
for all tuples of first-order formulas $\overline{\chi}$ in the signature of type $\overline{n^{\prime\prime}}$, and all tuples of first-order formulas $\overline{\xi}$ in the signature of type $\overline{n^{\prime}}$. Let $\mathcal{A}$ be the Henkin model of $\mathsf{ACA_0^{wo}}$ such that $A=M$ and $S_n[A]=\mathrm{Def}(M^n)$, where this is the set of subsets of $M^n$ which are definable by a first-order formula in $\mathcal{M}$ with parameters. Then $\overline{R}$, as interpreted in $\mathcal{M}$, determines an element of $S_{\,\overline{n}\,}[A]$, which we likewise refer to simply as $\overline{R}$. There are two cases:
\begin{itemize}[leftmargin=*]
\item First suppose that $\mathcal{A}\models Q(\overline{R})$. Then $\mathcal{A}\models \exists \; \overline{S} \; \Phi(\overline{R},\overline{S})$. But the witnesses $\overline{S}$ are elements of $S_{\,\overline{n^{\prime}}\,}[A] = \mathrm{Def}(M^{\overline{n^{\prime}}})$. But then $\overline{S}=\{\overline{v}: \mathcal{M}\models \overline{\xi}(\overline{v},\overline{a})\}$ for a tuple of first-order formulas $\overline{\xi}(\overline{v},\overline{x})$ in the signature of type $\overline{n^{\prime}}$ and parameters $\overline{a}$ from $M$. But then $\mathcal{A}\models \Phi(\overline{R},\overline{S})$ implies $\mathcal{M}\models \Phi(\overline{R}, \overline{\xi}(\cdot, \overline{a}))$, contradicting that $\mathcal{M}\models \forall \; \overline{x} \; \neg \Phi(\overline{R}, \overline{\xi}(\cdot, \overline{x}))$.
\item Second suppose that $\mathcal{A}\models \neg Q(\overline{R})$. Then $\mathcal{A}\models \exists \; \overline{S^{\prime}} \; \neg \Psi(\overline{R},\overline{S^{\prime}})$. But the witnesses $\overline{S^{\prime}}$ are elements of $S_{\,\overline{n^{\prime\prime}}\,}[A] = \mathrm{Def}(M^{\overline{n^{\prime\prime}}})$. But then $\overline{S^{\prime}}=\{\overline{v} : \mathcal{M}\models \overline{\chi}(\overline{v},\overline{b})\}$ for some tuple of first-order formulas $\overline{\chi}(\overline{v}, \overline{y})$ in the signature of type $\overline{n^{\prime\prime}}$ and parameters $\overline{b}$ from $M$. But then $\mathcal{A}\models \neg\Psi(\overline{R},\overline{S^{\prime}})$ implies $\mathcal{M}\models \neg \Psi(\overline{R}, \overline{\chi}(\cdot, \overline{b}))$, contradicting that $\mathcal{M}\models \forall \; \overline{y} \; \Psi(\overline{R}, \overline{\chi}(\cdot, \overline{y}))$.
\end{itemize}
In either case, we get a contradiction. Hence, rather (\ref{eqn:whattproves}) is true, for  some finite number of tuples of first-order formulas $\overline{\xi_1}, \ldots, \overline{\xi_n}$ in the signature of type $\overline{n^{\prime}}$.

Then, by the definition of $T$ and by compactness, there is a finite number of tuples of first-order formulas $\overline{\chi_1}, \ldots, \overline{\chi_m}$ in the signature of type $\overline{n^{\prime\prime}}$ such that 
\begin{equation}\label{eqn:whattproves2}
T_0\vdash \bigwedge_{j=1}^m \forall \; \overline{y} \; \Psi(\overline{R}, \overline{\chi_j}(\cdot, \overline{y})) \rightarrow  \bigvee_{i=1}^n \exists \; \overline{x} \; \Phi(\overline{R}, \overline{\xi_i}(\cdot, \overline{x}))
\end{equation}
Let $\overline{\xi^{\ast}}$ be the result of replacing the relation symbols $\overline{R}$ by second-order variables $\overline{R}$ in $\overline{\xi}$. Likewise, let $\overline{\chi^{\ast}}$ be the result of replacing the relation symbols $\overline{R}$ by second-order variables $\overline{R}$ in $\overline{\chi}$. Under this replacement, (\ref{eqn:whattproves2}) holds when we further replace provability in first-order logic relative to $T_0$ by provability in second-order logic relative to $\mathsf{ACA_0^{wo}}$. Finally, we claim:
\begin{equation}
\mathsf{ACA_0^{wo}} \vdash \forall \; \overline{R} \; \big(Q(\overline{R}) \leftrightarrow \bigwedge_{j=1}^m \forall \; \overline{y} \; \Psi(\overline{R}, \overline{\chi_j^{\ast}}(\cdot, \overline{y})) \big)
\end{equation}
The forward direction is trivial, using the $\Pi^1_1$-definition of $Q(\overline{R})$ and first-order comprehension. For the reverse direction, we simply use (\ref{eqn:whattproves2}) and the $\Sigma^1_1$-definition of $Q(\overline{R})$ and first-order comprehension.

\end{proof}

For the proof of the next theorem, we just need to recall one elementary definition from first-order model theory (cf. \cite[91]{Hodges1993-ux}):
\begin{defn}\label{defn:skolem} (Definable Skolem functions).

A model $\mathcal{M}$ has \emph{definable Skolem functions} if for every formula $\varphi(\overline{x}, y)$ with $\overline{x}$ having length $n$ there is a definable function $f:M^n\rightarrow M$ such that
\begin{equation*}
\mathcal{M}\models \big(\forall \; \overline{x} \; \exists \; y \; \varphi(\overline{x},y)\big)\rightarrow \big( \forall \; \overline{x} \; \varphi(\overline{x},f(\overline{x}))\big)
\end{equation*}
\end{defn}
The simple example which we use is:
\begin{ex}\label{ex:skolem} (Simple example of definable Skolem functions).

Any model of the first-order theory in a single binary relation $\preceq$ which says schematically that $\preceq$ is a well-order has definable Skolem functions. For, given $\varphi(\overline{x},y)$ one defines $f(\overline{x})=y$ iff $\varphi(\overline{x},y)\wedge \forall \; z\prec y \; \neg \varphi(\overline{x},z)$.
\end{ex}

\begin{rmk} (Extending to obtain definable Skolem functions with Cartesian products as co-domain).

If $\mathcal{M}$ is a first-order structure, then a definable function $f:M^n\rightarrow M^m$ is understood simply to be an $m$-tuple of definable functions $f_1:M^n\rightarrow M$, \ldots, $f_m:M^n\rightarrow M$ defined by $f(x_1, \ldots, x_n)=(f_1(x_1, \ldots, x_n), \ldots, f_m(x_1, \ldots, x_n))$. If a structure $\mathcal{M}$ has definable Skolem functions, then for every formula $\varphi(\overline{x}, \overline{y})$ with $\overline{x}$ having length $n$ and $\overline{y}$ having length $m$ there is a definable function $f:M^n\rightarrow M^m$ such that
\begin{equation*}
\mathcal{M}\models \big(\forall \; \overline{x} \; \exists \; \overline{y} \; \varphi(\overline{x},\overline{y})\big)\rightarrow \big( \forall \; \overline{x} \; \varphi(\overline{x},f(\overline{x}))\big)
\end{equation*}
This is simply by induction on $m\geq 1$. The base case holds by definition. Suppose $\mathcal{M}\models \forall \; \overline{x} \; \exists \; \overline{y} \; \exists \; z \; \varphi(\overline{x},\overline{y},z)$. Then by induction hypothesis there is definable $f:M^n\rightarrow M^m$ such that $\mathcal{M}\models \forall \; \overline{x} \; \exists \; z \; \varphi(\overline{x},f(\overline{x}),z)$. Then by the case $m=1$ there is a definable function $g:M^n\rightarrow M$ such that $\mathcal{M}\models \forall \; \overline{x} \; \varphi(\overline{x},f(\overline{x}),g(\overline{x}))$.
\end{rmk}

The following conservation result was originally developed by Barwise-Schlipf \cite{Barwise1975-xg} in the setting of arithmetic (see \cite[Lemma IX.4.3]{Simpson2009-ra}). It was deployed in the context of abstraction principles by Ferreira-Wehmeier \cite{Ferreira2002-xl} and later in Walsh \cite{walsh-comparing}. The proof is similar to the proof of \cite[Theorem 63(iv)]{walsh-comparing}. This proof does not obviously go through if one replaces $\mathsf{\Sigma^1_1\mbox{-}AC_0^{wo}}$ by $\mathsf{\Sigma^1_1\mbox{-}AC_0^{-}}$ and $\mathsf{ACA_0^{wo}}$ by $\mathsf{ACA_0^{-}}$, since this would remove the definable Skolem functions.

\begin{thm}\label{thm:conservative} (Conservation of $\mathsf{\Sigma^1_1\mbox{-}AC_0^{wo}}$ over $\mathsf{ACA_0^{wo}}$).

$\mathsf{\Sigma^1_1\mbox{-}AC_0^{wo}}$ is $\Pi^1_2$-conservative over $\mathsf{ACA_0^{wo}}$.
\end{thm}
\begin{proof}
Suppose that we have $\Pi^1_2$-sentence $\forall \; \overline{R} \; \exists \; \overline{R^{\prime}} \; \Phi(\overline{R}, \overline{R^{\prime}})$ where $\Phi(\overline{R}, \overline{R^{\prime}})$ contains no second-order quantifiers, and where $\overline{R}$ is of type $\overline{n}$, and $\overline{R^{\prime}}$ is of type $\overline{n^{\prime}}$. Suppose that $\mathsf{\Sigma^1_1\mbox{-}AC_0^{wo}}$ proves the sentence but  $\mathsf{ACA_0^{wo}}$ does not. Then there is a Henkin model $\mathcal{A}$ of $\mathsf{ACA_0^{wo}}$  such that $\mathcal{A}\models \forall \; \overline{R^{\prime}} \; \neg \Phi(\overline{R}, \overline{R^{\prime}})$, for some $\overline{R}$ in $\mathcal{A}$. Since each finite model of $\mathsf{ACA_0^{wo}}$ is standard and hence a model of $\mathsf{\Sigma^1_1\mbox{-}AC_0^{wo}}$, we must have that $\mathcal{A}$ is infinite. Without loss of generality, $\mathcal{A}$ is countably infinite. 

Let $\mathcal{M}$ be the associated first-order model with domain $A$ and with $n$-ary constant symbols interpreting each of the elements of $S_n[A]$. This includes the binary relation $\preceq$ which witnesses that $\mathcal{A}$ satisfies the axioms of a well-order. By first-order comprehension in $\mathcal{A}$, one has $S_n[A]=\mathrm{Def}(M^n)$. Hence for a tuple of first-order formulas $\overline{\xi}$ in this signature of type $\overline{n^{\prime}}$, one has $\mathcal{M}\models \forall \; \overline{x} \; \neg \Phi(\overline{R},\overline{\xi}(\cdot, \overline{x}))$. Further, $\mathcal{M}$ satisfies the schematic first-order theory expressing that $\preceq$ is a well-order in its signature. Because of this, $\mathcal{M}$ has definable skolem functions. 

Since $\mathcal{M}$ is infinite, choose a recursively saturated elementary extension $\mathcal{N}$ of $\mathcal{M}$. Let $\mathcal{B}$ be the Henkin model of $\mathsf{ACA_0^{wo}}$ with $B=N$ and $S_n[B]=\mathrm{Def}(N^n)$. Then $\mathcal{B}\models \forall \; \overline{R^{\prime}} \; \neg \Phi(\overline{R}, \overline{R^{\prime}})$, where $\overline{R}$ is the element of $S_{\,\overline{n}\,}[B]$ corresponding to the interpretation of the eponymous constants $\overline{R}$ as interpreted by $\mathcal{N}$. It suffices to show that $\mathcal{B}$ is a model of $\mathsf{\Sigma^1_1\mbox{-}AC_0^{wo}}$. For, then it would be a model of $\mathsf{\Sigma^1_1\mbox{-}AC_0^{wo}}$ plus the negation of the $\Pi^1_2$ sentence $\forall \; \overline{R} \; \exists \; \overline{R^{\prime}} \; \Phi(\overline{R}, \overline{R^{\prime}})$.

To show that $\mathcal{B}$ is a model of $\mathsf{\Sigma^1_1\mbox{-}AC_0^{wo}}$, suppose $\mathcal{B}\models \forall \; \overline{x} \; \exists \; \overline{R} \; \Psi(\overline{x}, \overline{R})$, where $\Psi(\overline{x}, \overline{R})$ has no second-order quantifiers, and where $\overline{R}$ has type $\overline{n}$. Then $\mathcal{N}\models \forall \; \overline{x} \; \bigvee_{\overline{\xi}} \;\exists \; \overline{y} \; \Psi(\overline{x}, \overline{\xi}(\cdot, \overline{y}))$, where $\overline{\xi}$ ranges over all tuples of first-order formulas in the signature of type $\overline{n}$. Then $\mathcal{N}$ does not realize the recursive class of formulas $\Gamma(\overline{x})$, which consists of all formulas $\neg \exists \; \overline{y} \; \Psi(\overline{x}, \overline{\xi}(\cdot, \overline{y}))$ as $\overline{\xi}$ ranges over first-order formulas in the signature of type $\overline{n}$. Since $\mathcal{N}$ is recursively saturated, $\Gamma(\overline{x})$ is not a type, i.e. not finitely realized in~$\mathcal{N}$. Hence, choose a finite sequence of formulas $\overline{\xi_1}, \ldots, \overline{\xi_{\ell}}$ of type $\overline{n}$ such that  $\mathcal{N}\models \forall \; \overline{x} \; \bigvee_{i=1}^{\ell} \;\exists \; \overline{y} \; \Psi(\overline{x}, \overline{\xi_i}(\cdot, \overline{y}))$. By padding with extra dummy variables if need be, we can assume that the $\overline{y}$ in each of these is of the same length. Then
\begin{equation}
\mathcal{N}\models \forall \; \overline{x} \; \exists \; \overline{y} \; \bigg(\bigvee_{i=1}^{\ell} \;\Psi(\overline{x}, \overline{\xi_i}(\cdot, \overline{y})) \wedge \bigwedge_{j=1}^{i-1} \neg  \exists \; \overline{z} \; \Psi(\overline{x}, \overline{\xi_j}(\cdot, \overline{z}))  \bigg)
\end{equation}
Since $\mathcal{N}$ has definable Skolem functions, there is function $f:N^{\left|\overline{x}\right|}\rightarrow N^{\left|\overline{y}\right|}$ which is $\mathcal{N}$-definable such that:
\begin{equation}
\mathcal{N}\models \forall \; \overline{x} \; \bigg(\bigvee_{i=1}^{\ell} \;\Psi(\overline{x}, \overline{\xi_i}(\cdot, \overline{f}(\overline{x}))) \wedge \bigwedge_{j=1}^{i-1} \neg  \exists \; \overline{z} \; \Psi(\overline{x}, \overline{\xi_j}(\cdot, \overline{z}))\bigg)
\end{equation}
Then we define a partition of $N^{\left|\overline{x}\right|}$, into definable subsets $P_1, \ldots, P_{\ell}$ as follows:
\begin{equation}
P_i = \{ \overline{x}: \mathcal{N}\models  \Psi(\overline{x}, \overline{\xi_i}(\cdot, \overline{f}(\overline{x}))) \wedge \bigwedge_{j=1}^{i-1} \neg  \exists \; \overline{z} \; \Psi(\overline{x}, \overline{\xi_j}(\cdot, \overline{z}))\}
\end{equation}
Then for all $1\leq j\leq \left|\overline{n}\right|$ we define:
\begin{equation}
R_j^{\prime} = \{ \overline{x}\,\overline{v_j} : \mathcal{N}\models \bigwedge_{1\leq i\leq \ell} \big(P_i\overline{x}\rightarrow \xi_{i,j}(\overline{v_j},f(\overline{x}))\big)\}
\end{equation}
This implies that for all $1\leq i\leq \ell$ and $1\leq j\leq \left|\overline{n}\right|$ we have:
\begin{equation}
\overline{x}\in P_i \Longrightarrow R_j^{\prime}[\overline{x}] = \{\overline{v_j}: \mathcal{N}\models \xi_{i,j}(\overline{v_j},f_j(\overline{x}))\}
\end{equation}
Then one has $\mathcal{B}\models \exists \; R_1^{\prime}, \ldots, R_k^{\prime} \; \forall \; \overline{x} \; \Psi(\overline{x}, R_1^{\prime}[\overline{x}], \ldots, R_k^{\prime}[\overline{x}])$, or in abbreviated form $\mathcal{B}\models \exists \; \overline{R^{\prime}} \; \forall \; \overline{x} \; \Psi(\overline{x}, \overline{R^{\prime}}[\overline{x}])$.
\end{proof}

This in place, we now prove Theorems~\ref{thm:fodefinable}-\ref{thm:fodefinabletwo} from \S\ref{sec:lower}:

\thmfodefinable*
\begin{proof}
Suppose that $Q$ is a generalized quantifier which is $\Delta^1_1$-definable relative to $\mathsf{\Sigma^1_1\mbox{-}AC_0^{wo}}$. Suppose that the $\Sigma^1_1$-formula is $\exists \; \overline{S} \; \Phi(\overline{R},\overline{S})$ and the $\Pi^1_1$-formula is  
$\forall \; \overline{S^{\prime}} \; \Psi(\overline{R},\overline{S^{\prime}})$, where $\Phi, \Psi$ have no second-order quantifiers. Then $\mathsf{\Sigma^1_1\mbox{-}AC_0^{wo}}$ proves each of the following:
\begin{align}
& \forall \; \overline{R} \; \big(\big(\exists \; \overline{S} \; \Phi(\overline{R}, \overline{S})\big)\rightarrow \big(\forall \; \overline{S^{\prime}} \; \Psi(\overline{R}, \overline{S^{\prime}})\big)  \big) \notag \\
& \forall \; \overline{R} \; \big(  \big(\forall \; \overline{S^{\prime}} \; \Psi(\overline{R}, \overline{S^{\prime}})\big) \rightarrow \big(\exists \; \overline{S} \; \Phi(\overline{R}, \overline{S})\big) \big) \label{eqn:theequiv}
\end{align}
We may assume that $\overline{S}$ and $\overline{S^{\prime}}$ are distinct, and that $\overline{S}$ does not occur free or bound in $\Psi(\overline{R}, \overline{S^{\prime}})$, and similarly $\overline{S^{\prime}}$ does not occur free or bound in $\Phi(\overline{R}, \overline{S})$. Then by elementary manipulation of the quantifiers this is equivalent to the following, the first of which is $\Pi^1_1$ and the second of which is $\Pi^1_2$:
\begin{align}
& \forall \; \overline{R} \; \forall \; \overline{S} \;\forall \; \overline{S^{\prime}} \; \big(\neg \Phi(\overline{R}, \overline{S})\vee \Psi(\overline{R}, \overline{S^{\prime}})\big) \notag \\
& \forall \; \overline{R} \; \exists \; \overline{S} \; \exists \; \overline{S^{\prime}}   \big(\neg \Psi(\overline{R}, \overline{S^{\prime}})\vee \Phi(\overline{R}, \overline{S})\big) \label{eqn:theequiv:2}
\end{align}
Since $\mathsf{\Sigma^1_1\mbox{-}AC_0^{wo}}$ is $\Pi^1_2$-conservative over $\mathsf{ACA_0^{wo}}$ by Theorem~\ref{thm:conservative}, one has that $\mathsf{ACA_0^{wo}}$ proves (\ref{eqn:theequiv:2}), and hence also (\ref{eqn:theequiv}). Let $Q^{\prime}(\overline{R})$ be the $\Sigma^1_1$-formula $\exists \; \overline{S} \; \Phi(\overline{R},\overline{S})$. Then $Q^{\prime}(\overline{R})$  is $\Delta^1_1$-definable relative to $\mathsf{ACA_0^{wo}}$. Then by Theorem~\ref{thm:fromdeltatoaca}, we have that $Q^{\prime}(\overline{R})$ is first-order definable relative to $\mathsf{ACA_0^{wo}}$. Then since $\mathsf{\Sigma^1_1\mbox{-}AC_0^{wo}}$ proves that $Q(\overline{R})$ and $Q^{\prime}(\overline{R})$ are equivalent, we have that  $Q(\overline{R})$ is first-order definable relative to  $\mathsf{\Sigma^1_1\mbox{-}AC_0^{wo}}$.

\end{proof}

\thmfodefinabletwo*
\begin{proof}
This follows from Theorem~\ref{thm:fefermanesq} and Theorem~\ref{thm:fodefinable} and Theorem~\ref{thm:fefermanesqfourpre}.
\end{proof}

\section{Proof of Theorem~\ref{thmmostdefinabel}}\label{sec:cardinalityqs}

Our goal in this section is to prove Theorem~\ref{thmmostdefinabel}. 

First we introduce some notation specific to $\scompfin$, which recall was defined at the end of Definition~\ref{defn:two}.

We define $X\upharpoonright a = \{x\in X: x\prec a\}$, and define $[0,a)=\{x: x\prec a\}$, both of which exist by first-order comprehension.  Define $X\cong Y$ if there is order isomorphism from $(X,\preceq)$ to $(Y,\preceq)$, which is a $\Sigma^1_1$-notion. By an easy inductive argument, order isomorphisms when they exist are unique.

An axiom of $\scompfin$ is that the universe $\mathbb{V}$ is Dedekind finite (cf. Definition~\ref{defn:cardinality} for the definition of Dedekind finite). It follows from this that every $X$ is Dedekind finite, since any injection $f:X\rightarrow X$ which is not a surjection can be extended to an injection $g:\mathbb{V}\rightarrow \mathbb{V}$ which is not a surjection by setting $g=f$ on $X$ and $g$ equal to the identity on $\mathbb{V}\setminus X$.

Then we have:
\begin{prop}\label{prop:t0proves} (Order isomorphisms in $\scompfin$).

$\scompfin$ proves that for all $X$ either there is $c$ such that $X \cong [0,c)$, or $X=\mathbb{V}$, but not both.
\end{prop}
\begin{proof}
We work in $\scompfin$. In this proof, we use $0$ as an abbreviation for the least element, and we use $b+1$ as an abbreviation for the immediate successor of $b$ when it exists. Using the well-ordering, note that $b+1$ exists whenever there is $c\succ b$.

We make four preliminary observations.

First, since the universe is Dedekind finite, every non-empty set $X$ has a maximal element. For, otherwise, for all $a$ in $X$ there is $b$ in $X$ with $b\succ a$. Hence, for all $a$ in $X$ there is $\preceq$-least $b$ in $X$ with $b\succ a$. Let $R$ be the binary relation which collects together pairs $(a,b)$ in $X^2$ such that $b\succ a$ and $b$ is least with this property, and note that $R$ exists by first-order comprehension. Then $R$ is the graph of an injective non-surjective function $f:X\rightarrow X$:
\begin{itemize}[leftmargin=*]
    \item It is non-surjective since $0$ is not in its range.
    \item It is injective since $a\prec b$  with both $a,b$ in $X$ implies $f(a)\preceq b \prec f(b)$. 
\end{itemize}

Second, since the universe $\mathbb{V}$ is Dedekind finite, every $a$ is either $0$ or $a=b+1$ for some $b$. This is trivial in the case that $\mathbb{V}=\{0\}$. Hence we suppose that we are not in this case. Let $X=\{a: a=0\vee \exists \; b \; a=b+1\}$, which exists by first-order comprehension. By induction we show that $Xa$ holds for all $a$. Suppose not. Then there is some $\preceq$-least $a$ for which it does not hold. Then $a$ cannot be zero and $a$ is not the immediate successor of anything. Then $[0,a)$ is non-empty and so by the previous paragraph has a greatest element, call it $b$. Since $b\prec a$, one has both that $b+1$ exists and $b+1\preceq a$. But by maximality of $b$, we have that $b+1$ cannot be in $[0,a)$, and so $a\preceq b+1$. Then $a=b+1$, a contradiction.

Third, to show that all elements are in a set $Y$, it suffices to show both of the following:
\begin{itemize}[leftmargin=*]
    \item \emph{Base case}: $Y(0)$
    \item \emph{Induction step}: if $Y(b)$ and $b+1$ exists then $Y(b+1)$.
\end{itemize}
For, suppose that these two things hold and yet some $a$ satisfies $\neg Y(a)$. Then choose $\preceq$-least $a$ such that $\neg Y(a)$. By the previous paragraph $a=0$ or $a=b+1$ for some $b$. But if the former, then we contradict the base case, and if the latter then we contradict the induction step.

Fourth, for every $X$ and every $a$ there is a $c\preceq a$ such that $X\upharpoonright a \cong [0,c)$. For, fix $X$ and do an induction on $a$ with respect to the $\Sigma^1_1$-definable property of there being a $c\preceq a$ with $X\upharpoonright a \cong [0,c)$:
\begin{itemize}[leftmargin=*]
    \item If $a$ is $0$, we take $c=a$. 
    \item If $a=b+1$ and $X\upharpoonright b\cong [0,d)$ for some $d\preceq b$, then if $b$ in $X$ we have $X\upharpoonright a \cong [0,d+1)$ (note that $d+1$ exists since $d\preceq b$ and $b+1$ exists); and if $b$ not in $X$ we have $X\upharpoonright a \cong [0,d)$.
\end{itemize}

We now prove that for every $X$, either there is $c$ such that $X \cong [0,c)$ or $X=\mathbb{V}$. Let $a$ be the maximal element of $\mathbb{V}$. Then $X\upharpoonright a\cong [0,c)$ for some $c\preceq a$, by our previous work. There are then a couple of cases to consider.
\begin{itemize}[leftmargin=*]
    \item  If $a$ not in $X$, then $X\cong [0,c)$.
    \item  If $a$ in $X$, then consider two cases: $c\prec a$ or $c=a$.
\begin{itemize}[leftmargin=*]
    \item If $c\prec a$, then $c+1$ exists and $X\cong [0,c+1)$.
    \item If $c=a$, then we can extend an order isomorphism $f:X\upharpoonright a\rightarrow [0,c)$ to an order isomorphism $g:X\rightarrow \mathbb{V}$ by sending $g(a)=c$. Then by considering its inverse $g^{-1}:\mathbb{V}\rightarrow X$, we have that $g^{-1}:\mathbb{V}\rightarrow \mathbb{V}$ is injective and hence surjective, and thus $X=\mathbb{V}$.
\end{itemize}    
\end{itemize}

Finally, one cannot have both $X=\mathbb{V}$ and $X \cong [0,c)$ for some $c$. For, any order isomorphism $f:\mathbb{V}\rightarrow [0,c)$ would be an injective non-surjective function $f:\mathbb{V}\rightarrow \mathbb{V}$.
\end{proof}

This implies:
\begin{prop}\label{prop:t0proves3} (A cardinality dichotomy in $\scompfin$).

$\scompfin$ proves that for all $X,Y$ one has that there is an injection from $X$ to $Y$, or an injection which is not a surjection from $Y$ to $X$, but not both.
\end{prop}
\begin{proof}
One cannot have both, since if $f:X\rightarrow Y$ is an injection and $g:Y\rightarrow X$ is an injection which is not a surjection, then $g\circ f:X\rightarrow X$ is an injection which is not a surjection. 

If $Y$ is $\mathbb{V}$, then we can satisfy the first disjunct with the identity map. If $Y$ is not $\mathbb{V}$ and $X$ is $\mathbb{V}$, then we can satisfy the second disjunct with the identity map.

Suppose now that both $X,Y$ are not $\mathbb{V}$. By Proposition~\ref{prop:t0proves}, we have that $X\cong [0,c)$ and $Y\cong [0,d)$ for some $c,d$, with witnessing order isomorphisms $f,g$. If $c\preceq d$, then $g^{-1}\circ f$ is an injection from $X$ to $Y$. If $d\prec c$, then $f^{-1}\circ g$ is an injection which is not a surjection from $Y$ to $X$; in particular, a point in $X$ which is not in the range of $f^{-1}\circ g$ is the point $x$ of $X$ with $f(x)=d$.
\end{proof}

\begin{prop}\label{prop:t0proves4} (Equivalent characterizations of strict less-than of cardinality in $\scompfin$).

$\scompfin$ proves that for all $X,Y$ the following are equivalent:
\begin{enumerate}[leftmargin=*]
    \item\label{prop:t0proves41} There is an injection $f:X\rightarrow Y$ but there is no injection $g:Y\rightarrow X$.
    \item\label{prop:t0proves42} There is an injection $f:X\rightarrow Y$ which is not a surjection.
    \item\label{prop:t0proves43} There is no surjection $h:X\rightarrow Y$.
\end{enumerate}
\end{prop}
\begin{proof}
First suppose (\ref{prop:t0proves41}); we show (\ref{prop:t0proves42}).  Suppose that there is an injection $f:X\rightarrow Y$ but there is no injection $g:Y\rightarrow X$. Then we claim that $f:X\rightarrow Y$ is not a surjection. If it was, then its inverse $f^{-1}:Y\rightarrow X$ would be an injection.

Second suppose (\ref{prop:t0proves42}); we show (\ref{prop:t0proves43}). Suppose that $f:X\rightarrow Y$ is an injection but not a surjection. Suppose for reductio that there was a surjection $h:X\rightarrow Y$. Then consider the binary relation $R$ consisting of pairs $(y,x)$ in $Y\times X$ such that $h(x)=y$ and there is no $x_0\prec x$ with $h(x_0)=y$. The binary relation $R$ exists by first-order comprehension. Then $R$ defines the graph of an injective function $g:Y\rightarrow X$. But then with $f,g$ we have both disjuncts of Proposition~\ref{prop:t0proves3}.

Third suppose (\ref{prop:t0proves43}); we show (\ref{prop:t0proves41}). If we had an injection $g:Y\rightarrow X$, then we could choose a point $y_0$ of $Y$, and consider the binary relation $R$ consisting of pairs $(x,y)$ in $X\times Y$ such that if $x$ is in the range of $g$, then $g(y)=x$, while if $x$ is not in the range of $g$, then $y=y_0$. The binary relation $R$ exists by first-order comprehension. Then $R$ defines the graph of a surjective function $h:X\rightarrow Y$. Hence, there is no injection $g:Y\rightarrow X$. Then by Proposition~\ref{prop:t0proves3}, there is an injection $f:X\rightarrow Y$ which is not a surjection. Hence, there is an injection $f:X\rightarrow Y$, but there is no injection $g:Y\rightarrow X$.
\end{proof}

Finally, we can prove:
\thmmostdefinabel*
\begin{proof}
For (\ref{thmmostdefinabel:1}), it suffices to note that the relation $\left|X\right|< \left|Y\right|$ is bounded $\Delta^1_1$-definable in $\scompfin$. Recall from \S\ref{sec:intro} that we define $\left|X\right|< \left|Y\right|$ in second-order logic as: there is an injection $f:X\rightarrow Y$ but there is no injection $g:Y\rightarrow X$. This condition is just what Proposition~\ref{prop:t0proves4}(\ref{prop:t0proves41}) says. And Proposition~\ref{prop:t0proves4}(\ref{prop:t0proves42}) provides the bounded $\Sigma^1_1$-definition, and Proposition~\ref{prop:t0proves4}(\ref{prop:t0proves43}) provides the bounded $\Pi^1_1$-definition. 

For (\ref{thmmostdefinabel:2}), to show Henkin domain independence, we simply appeal to (\ref{thmmostdefinabel:1}) and Corollary~\ref{cor:stable}(\ref{cor:stable:3}). Finally, to see that $\mathsf{Most}$ and $\mathsf{More\mbox{-}then}$ are Henkin conservative, note that this is true for the same reason that they are conservative on standard models. In the case of $\mathsf{Most}$ one simply notes that $F\cap G$ is the same as $F\cap (F\cap G)$, and that $F\setminus G$ is the same as $F\setminus (F\cap G)$ (cf. Example~\ref{ex:most}). In the case of $\mathsf{More\mbox{-}than}$ one simply notes that $F\cap H$ is the same as $F\cap ((F\cup G)\cap H)$, and that $G\cap H$ is the same as $G\cap ((F\cup G)\cap H)$ (cf. Example~\ref{ex:type111conserv}).
\end{proof}

\section{Further Questions}\label{sec:further}

In closing, four further issues bear investigating. 

First, following the tradition of generalized quantifiers, this paper works in second-order logic and with its models. But in many applications one will be working with $\omega$-th order logic or with a type theory. It is not clear whether any of the results of this paper extend to those natural settings. For, the tools from first-order model theory which are used all depend on being able to reduce the ``top layer'' second-order objects to various kinds of definable first-order formulas, in certain elementary extensions. These methods will break down if there is no top layer and if higher layers can influence lower layers.

Second, all of the proofs in this paper are highly model-theoretic, and go through the rich method of recursively saturated models initiated by Barwise-Schlipf.\footnote{\cite{Barwise1975-xg}.} However, as mentioned, Theorem~\ref{thm:fefermanesq} basically follows from Feferman's Preservation Theorem, which Feferman originally proved by proof-theoretic means. It would be valuable to go through the arguments of this paper using proof-theoretic methods. Perhaps these methods could provide more of a procedure for i.e. going from a generalized quantifier satisfying Henkin domain independence to its $\Delta^1_1$-definition. The arguments of this paper, by contrast, provide no such information since they proceed by non-constructive compactness arguments.

Third, it seems worthwhile to explore the idea of Henkin domain independence in the setting of a modal logic. One can view the approach adumbrated in this paper as the rudiments of a Kripke semantics for such a logic, whereby the worlds are the Henkin models and the accessibility relation is the substructure relation. As a first attempt, one might try to formalize principles such as the following, which seems to capture some part of Henkin upwards independence in the case of type~$(1,1)$ generalized quantifiers:
\begin{equation}\label{eqn:modalrigid}
\forall \; X \; \forall \; Y \; \big(Q(X,Y) \leftrightarrow \Box \;Q(X,Y)\big)
\end{equation}
A natural question would be what modal logic, if any, was sound and complete for the aforementioned semantics, and whether these resulted in (\ref{eqn:modalrigid}) or strengthenings thereof having implications for definability.

Fourth, it is worth examining closer the relation between Theorems~\ref{thm:fefermanesq}, \ref{thm:fefermanesqfour} and the Conservativity Theorem of Keenan-Stavi.\footnote{\cite[276, 317 ff]{Keenan1986-pq}, \cite[p. 248, Fact 5]{Gamut1991-gc}, \cite[305-6]{Glanzberg2006-is}, \cite[321]{Westerstahl2007-fb}.} This result of Keenan-Stavi gives a ``from below'' characterization of conservative quantifiers on standard domains, namely as those containing certain basic generalized quantifiers and closed under the Boolean operations and a relativisation notion.\footnote{At least, it is usually presented as a result about standard models. But it is not hard to see that with minimal modification it can be recast as a theorem of higher-order logic.} However, while an immensely elegant characterization, it is sometimes thought to depend too much on the underlying first-order domain, and to be ``mere enumeration.''\footnote{\cite[10-11]{van-Benthem1986-qu}; cf. \cite[305-6]{Glanzberg2006-is}, \cite[321]{Westerstahl2007-fb}.} But in other settings, $\Delta^1_1$-definitions have equivalent characterizations which do not depend heavily on the underlying domain and which give voice to some kind of construction from below. These can be hard-fought results, such as Souslin's Theorem, or the characterizations of $\Delta^1_1$-definable sets of natural numbers in terms of Kleene's~$\mathcal{O}$.\footnote{\cite[Chapters 4-5]{Ash2000-qv}. This was generalized by Barwise and others to the setting of admissible sets (e.g. \cite[\S{IV.3}]{Barwise1975-oe}).} In other cases, such as in $\mathsf{NP}\cap \mathsf{co\mbox{-}NP}$, there are to my knowledge no known ``from below'' characterizations. One natural question then is whether one can refine the Keenan-Stavi theorem to show that the bounded $\Delta^1_1$-definable generalized quantifiers are built up from below from simpler ones in a constructive manner which does not depend too heavily on the underlying domain.

\stoptoc

\section{Acknowledgments and funding}

\subsection{Acknowledgments}

Thanks to Dylan Bumford, Salvatore Florio, James Gu, Ed Keenan, Ned Sanger, and the referees and editors for feedback and comments.

\subsection{Funding}

There is no external funding to report.

\resumetoc

\bibliographystyle{amsplain}
\bibliography{bibliography.bib}

\end{document}